\RequirePackage{fix-cm}
\documentclass[smallcondensed]{svjour3}     
\smartqed  
\usepackage{graphicx}
\usepackage{booktabs}
\usepackage{amssymb}
\usepackage{enumitem}
\usepackage{url}
\usepackage{siunitx}
\usepackage{amsmath} 
\usepackage{bm}
\usepackage[acronym]{glossaries}   

\newacronym{pwl}{PWL}{piecewise linear}
\newacronym{cpwl}{CPWL}{continuous piecewise linear}
\newacronym{lp}{LP}{linear programming}
\newacronym{milp}{MILP}{mixed-integer linear programming}
\newacronym{dc}{DC}{difference-of-convex}
\newacronym{minlp}{MINLP}{mixed-integer nonlinear programs}
\newacronym{miqcp}{MIQCP}{mixed-integer quadratic constrained programming}
\newacronym{qp}{QP}{quadratic programming}
\newacronym{mc}{MC}{Multiple Choice}
\newacronym{scbp}{SCBP}{sequentially-coupled bilinear programs}

\makeatletter
\renewcommand\@biblabel[1]{\makebox[1.5em][r]{#1.}}
\makeatother

\begin{document}

\title{When MILP Beats QP: Piecewise-Linear Reformulations of Sequentially Coupled Bilinear Programs
}

\titlerunning{When MILP Beats QP}        

\author{Quentin Ploussard         \and
        Maris Usis \and
        Oluwabunmi Iwakin \and
        Matija Pavičević
}


\institute{Quentin Ploussard \at
              Argonne National Laboratory, Lemont, IL, USA \\
              Tel.: (630) 252-5149\\
              \email{qploussard@anl.gov}           
           \and
           Maris Usis \at
              Argonne National Laboratory, Lemont, IL, USA
           \and
           Oluwabunmi Iwakin \at
              Lehigh University, Bethlehem, PA, USA
            \and
            Matija Pavičević \at
              Argonne National Laboratory, Lemont, IL, USA
}

\date{Received: date / Accepted: date}

\maketitle

\begin{abstract}
The presence of bilinear terms in mathematical modeling generally yields nonconvex quadratic programs (QPs) that remain computationally challenging to solve to global optimality. While continuous piecewise linear (CPWL) approximations can reformulate these nonlinearities into mixed-integer linear programs (MILPs), the geometric construction of the domain partition heavily dictates the resulting solver efficiency. In this paper, we present efficient MILP formulations for approximating bilinear terms and rigorously evaluate their computational merits against direct QP solvers. First, we introduce CPWL approximations with arbitrary high degrees of accuracy that explicitly account for and exploit the inherent symmetries of the unit bilinear function. Second, we analytically establish the exact maximum and average approximation errors of the CPWL approximations. Third, we construct and compare three distinct MILP formulations (a ``Triangle'', a ``Square'', and a difference-of-convex ``DC'' formulation) of the CPWL functions. Fourth, we formalize a highly relevant class of optimization models, Sequentially Coupled Bilinear Programs (SCBP), where variables represent system states and state changes. Finally, through extensive numerical experiments, we demonstrate that our compact ``Square'' MILP formulation achieves superior computational performance on SCBP instances with long sequence horizons, allowing mature open-source MILP solvers to effectively outperform QP solvers.
\end{abstract}

\section{Introduction}
\label{intro}

The product of two continuous decision variables (commonly referred to as a bilinear term) is widely used in mathematical modeling to describe a diverse range of physical and economic relationships. This representation is fundamental in global optimization applications spanning generalized pooling problems that capture stream mixing and blending properties \cite{hao_modeling_2026,deassis_piecewise_2017}, flow-head dynamics in optimal hydropower scheduling and water distribution \cite{cerisola_stochastic_2012,guedes_unit_2017,tasseff_polyhedral_2024}, network flow optimization for transportation systems \cite{klansek_comparison_2014,gao_piecewise_2018}, and production planning \cite{yan_hierarchical_2025,gabrielli_optimization_2020}. However, the presence of bilinear terms generally yields nonconvex \gls{qp} or \gls{miqcp} problems. Because these formulations require continuous spatial branch-and-bound algorithms \cite{bestuzheva_global_2025}, they often suffer from weak relaxations and scale poorly, remaining computationally challenging even for realistically sized instances.

Recent advances in algorithmic design have enabled general-purpose global solvers to tackle these nonconvex \gls{qp}s by combining spatial branch-and-bound algorithms with convex relaxation techniques \cite{grubel_successive_2023,linan_trends_2025,misener_apogee_2011,gurobi_optimization_2024}. The most notable of these are McCormick envelopes, which construct polyhedral over- and under-estimators to establish rigorous bounds \cite{mccormick_computability_1976,liberti_exact_2006}. The computational efficiency of these solvers is largely governed by the tightness of these relaxations \cite{nagarajan_tightening_2016}. To tighten the bounds, piecewise McCormick relaxation methods often partition the variable domains \cite{wicaksono_piecewise_2008,hasan_piecewise_2010}. Unfortunately, traditional spatial partitioning strategies can quickly yield prohibitively large models, limiting practical tractability \cite{castro_tightening_2015}.

In this study, we place particular emphasis on leveraging \gls{milp} solvers to address nonconvex \gls{qp}s by employing \gls{cpwl} approximations of the bilinear terms. This paradigm shift is strongly motivated by the remarkable maturity of state-of-the-art \gls{milp} solvers, including open-source tools \cite{Huangfu2018Parallelizing}, which can often process discrete branch-and-bound trees more efficiently than their continuous nonlinear counterparts. Recent literature strongly supports this direction; for example, Zhang et al. \cite{zhang_improving_2025} recently demonstrated that progressive \gls{milp} methods can significantly improve the solution of indefinite \gls{qp}s. Focusing specifically on geometric approximations, certain \gls{cpwl} functions can be represented through continuous relaxations that are as sharp as McCormick envelopes, facilitating tight bounds that preserve computational efficiency \cite{barmann_approximation_2023,geissler_using_2012}. Consequently, literature on optimal \gls{cpwl} approximations has grown significantly, highlighting triangulation methods \cite{kutzer_using_2021,vielma_mixed_2010,dambrosio_piecewise_2010,dambrosio_application_2010,borghetti_milp_2008,vielma_modeling_2011,vielma_embedding_2018,huchette_combinatorial_2019,huchette_nonconvex_2023,barmann_piecewise_2022} and \gls{dc} representations \cite{kripfganz_piecewise_1987,kazda_nonconvex_2021,kazda_linear_2024,ploussard_tightening_2025} as viable ways to capture bivariate functions.

Despite these advancements, optimizing the construction of \gls{cpwl} functions to yield the most tractable \gls{milp} remains an open research problem. Specifically, there is a lack of theoretical analyses establishing guaranteed approximation error bounds across different geometric partitions. More importantly, from a practical Operations Research perspective, the specific problem structures and conditions under which a \gls{cpwl}-approximated \gls{milp} formulation definitively outperforms a direct, state-of-the-art \gls{qp} solver remain largely undefined.

To address these gaps, we present exact, efficient \gls{milp} formulations for approximating bilinear terms and rigorously evaluate their computational merits against direct \gls{qp} solvers. The main contributions of this paper are summarized as follows:
\begin{itemize}
    \item First, we formally introduce the \gls{cpwl} function $g_n$, designed to approximate the unit bilinear term $[0,1]^2 \rightarrow \mathbb{R}, \enspace (x,y) \rightarrow xy$ with an arbitrary, predefined degree of accuracy $n$. Crucially, the geometric construction of this approximation explicitly accounts for and exploits the inherent symmetries of the bilinear function.
    
    \item Second, we analytically establish the exact maximum and average approximation errors of $g_n$.
    
    \item Third, we construct and compare three distinct \gls{milp} formulations of the \gls{cpwl} function $g_n$ : a ``Triangle'' formulation with $4n^2$ pieces, a ``Square'' formulation with $2n(n+1)$ pieces, and a ``\gls{dc}'' formulation comprised of $4n$ pieces.
    
    \item Fourth, we define a specific, highly relevant class of optimization models, \gls{scbp}, where the objective function consists of a sum of bilinear terms, and each bilinear term is composed of one variable representing the system state, and another variable representing the change in state.
    
    \item Finally, through extensive numerical experiments, we demonstrate that while general quadratic formulations (such as those in QPLIB \cite{Furini2019QPLIB}) are best left to \gls{qp} solvers, the ``Square'' \gls{milp} formulation achieves superior computational performance on \gls{scbp} instances with long sequence horizons, allowing mature open-source \gls{milp} solvers to effectively beat commercial \gls{qp} solvers.
\end{itemize}

The remainder of this paper is organized as follows. In Section 2, we introduce the \gls{cpwl} approximation of the bilinear term and analyze its approximation error. Section 3 details the \gls{milp} formulations for the distinct geometric representations and discusses their size complexity. In Section 4, we formally define the class of \gls{scbp} problems and analyze the feasibility and quality of their \gls{cpwl} solutions. Section 5 presents extensive numerical experiments comparing the computational performance of our proposed \gls{milp} approximations against state-of-the-art direct \gls{qp} solvers on both \gls{scbp} and QPLIB instances. Finally,  we conclude the paper in Section 6, while highlighting viable future contributions.

\section{\gls{cpwl} approximation of the bilinear term}
\label{sec:cpwl_approx}
In this section, we introduce the \gls{cpwl} function $g_n$ that we build by triangulating a set of points in $[0,1]^3$. We call this triangulation the ``Triangle'' representation of $g_n$. Then, leveraging the geometry of the linear pieces of $g_n$, we introduce two alternative representations called ``Square'' and ``\gls{dc}'' representations. Additionally, we analytically establish the maximum and average approximation error between $g_n$ and the unit bilinear term.

\subsection{Expression of the \gls{cpwl} function}

Let $f : [0,1]^2 \to \mathbb{R}, \quad f(x, y) = x y$.
This section introduces the \gls{cpwl} function $g_n: [0,1]^2 \to \mathbb{R}$ designed to approximate $f$ with a certain degree of accuracy $n$. 


To construct the \gls{cpwl} approximation $g_n$, we evaluate $f$ across a regularly spaced grid and connect the resulting points via triangulation. Because $f$ is symmetric with respect to $x+y=0$ and $x-y=0$, the grid is naturally oriented along these lines. The grid of points is defined by the set $S_n$ below.
\[
S_n = \left\{ \left(\frac{j+k-n}{2n},\frac{j-k+n}{2n} \right) \in [0,1]^2, \quad j,k \in \{0, ..., 2n\}\right\}
\]

\begin{definition}
We define the function $g_n$ as the \gls{cpwl} function obtained by triangulating the set of points $\{(x,y,f(x,y)) \in [0,1]^3, \;(x,y) \in S_n\}$.
\end{definition}

An illustration of $S_n$ and its corresponding triangulation is depicted in Figure \ref{fig:domain}. Such triangulation results in a total of $4n^2$ triangular pieces. We call this \gls{cpwl} representation the ``Triangle'' representation of $g_n$.

\begin{figure}
    \centering
    \includegraphics[width=1.0\linewidth]{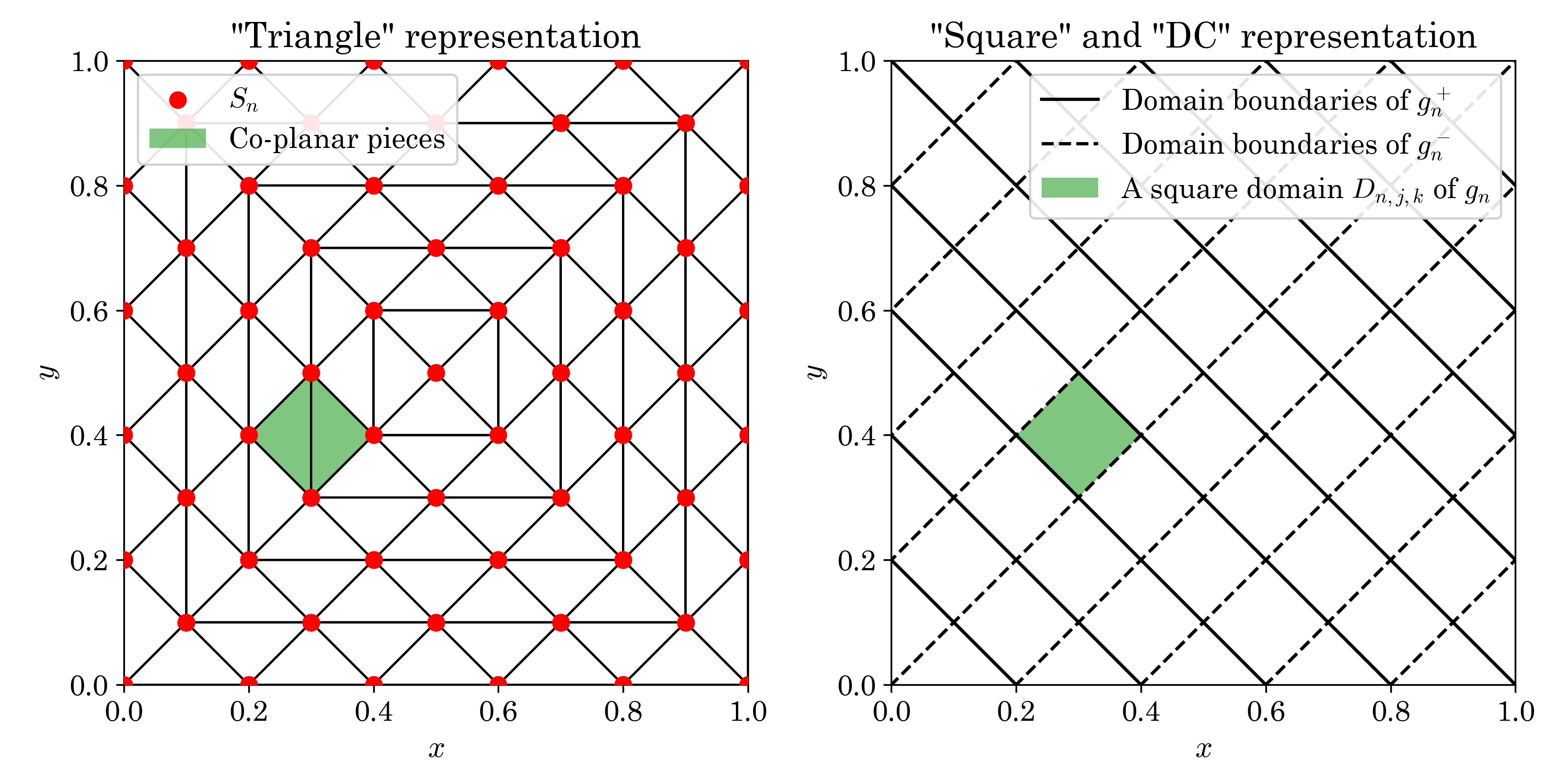}
    \caption[Linear domains of $g_n$ depending on its representation ($n=5$)]{Linear domains of $g_n$ depending on its representation ($n=5$)}
\label{fig:domain}
\end{figure}

An analysis of $g_n$'s gradient show that all interior triangular pieces can be grouped into pairs of adjacent co-planar pieces. Each of these pairs can be merged into a linear piece defined on a square region, thereby reducing the total number of linear pieces to $2n(n+1)$. We call this \gls{cpwl} representation the ``Square'' representation of $g_n$. This representation is illustrated in Figure \ref{fig:domain}. In the ``Square'' representation, the domains of the linear pieces of $g_n$ are:
\[
D_{n,j,k} = \left\{ (x,y) \in [0,1]^2, \left\{
\begin{array}{c}
\frac{j}{n} \leq x+y \leq \frac{j+1}{n}\ \\ \frac{k-n}{n} \leq y-x \leq \frac{k+1-n}{n} 
\end{array}
 \right\} \right\}, \quad j,k \in \{0,...,2n-1\}
\]

And the function $g_n$ can be expressed as
\begin{align*}
&g_n(x,y) = \\ 
&\left( \frac{j+k+1-n}{2n} \right) x + \left( \frac{j-k+n}{2n} \right) y  -
\left( \frac{j+k+1-n}{2n} \right)\left( \frac{j-k+n}{2n} \right)
\\ &\forall (x,y) \in D_{n,j,k}
\end{align*}

A proof of the co-planarity of the adjacent triangular pieces and of the expression of $g_n$ on $D_{n,j,k}$ can be found in Online Resource 1 (ESM\_1).

Furthermore, $g_n$ can be expressed as a \gls{dc} function $g_n = g_n^+ - g_n^-$. This decomposition is based on the on the separation of the variables $u$ and $v$ in $xy = u^2 - v^2$, where $u = \frac{1}{2}(x+y)$ and $v=\frac{1}{2}(x-y)$ \cite{barmann_piecewise_2022}. Each convex component can be expressed as a \gls{cpwl} function composed of $2n$ linear pieces. The domain of these linear pieces are strips of width $\frac{1}{n\sqrt{2}}$ oriented at a \ang{-45} and \ang{45} angle (Figure \ref{fig:domain}).

The convex components of $g_n$ can be expressed as:
\begin{align*}
&g_n^+(x,y)
= \max_{0 \leq j < 2n} g_{n,j}^+(x,y) \\
&= \max_{0 \leq j < 2n}
\left[
\left( \frac{2j+1}{4n}\right)x
+
\left( \frac{2j+1}{4n}\right)y
-
\frac{j(j+1)}{4n^2}
\right] \\
&g_n^-(x,y)
= \max_{0 \leq k < 2n} g_{n,k}^-(x,y) \\
&= \max_{0 \leq k < 2n}
\left[
-\left( \frac{2(k-n)+1}{4n} \right)x
+
\left( \frac{2(k-n)+1}{4n} \right)y
-
\frac{(k-n)(k-n+1)}{4n^2}
\right]
\end{align*}

It can be shown that $g_n^+$ is equal to $g_{n,j}^+$ on the domain $D_{n,j}^+=\{(x,y) \in [0,1]^2: \frac{j}{n} \leq x+y \leq \frac{j+1}{n}\}$, and that $g_n^-$ is equal to $g_{n,k}^-$ on the domain $D_{n,k}^-=\{(x,y) \in [0,1]^2: \frac{k}{n} \leq y-x+1 \leq \frac{k+1}{n}\}$. A detailed proof of the \gls{dc} representation of $g_n$ can be found in Online Resource 1 (ESM\_1).

\begin{remark}
Note that each ``Square'' domain corresponds to the intersection of a linear domain of $g_n^+$ and $g_n^-$, as illustrated in Figure \ref{fig:domain}. In other words, $D_{n,j,k} = D_{n,j}^+ \cap D_{n,k}^-$.
\end{remark}

A summary of the linear piece features of each representation is provided in Table \ref{tab:nber_pieces}. An illustration of the \gls{cpwl} function $g_n$ and of its \gls{dc} components is shown in Figure \ref{fig:3D_plot}. Note that the surfaces of the functions $g_n^+$ and $g_n^-$ connect seamlessly along the lines $x=0$ and $y=0$, which is consistent with the functions $g_n$ and $f$ being equal to zero on these lines.

\begin{remark}
It is worth emphasizing that the ``Triangle'', ``Square'', and ``\gls{dc}'' representations map to the exact same underlying \gls{cpwl} function \( g_n \). They differ fundamentally, however, in how they geometrically partition the variable domain into affine pieces, which consequently leads to distinct \gls{milp} formulations as described in the next section.
\end{remark}

\begin{table}[htbp]
\centering
\caption{Features of the linear pieces for each representation of $g_n$}
\label{tab:nber_pieces}
\begin{tabular}{p{2.4cm} p{3.2cm} p{2.6cm}}
\toprule
\textbf{Representation} & \textbf{Domain geometry of the linear pieces} & \textbf{Number of linear pieces} \\
\midrule
Triangle & Triangles & $4n^2$ \\
Square & Squares & $2n(n+1)$ \\
DC & Strips & $2 \times 2n$ \\
\bottomrule
\end{tabular}
\end{table}

\begin{figure}
    \centering    \includegraphics[width=1.0\linewidth]{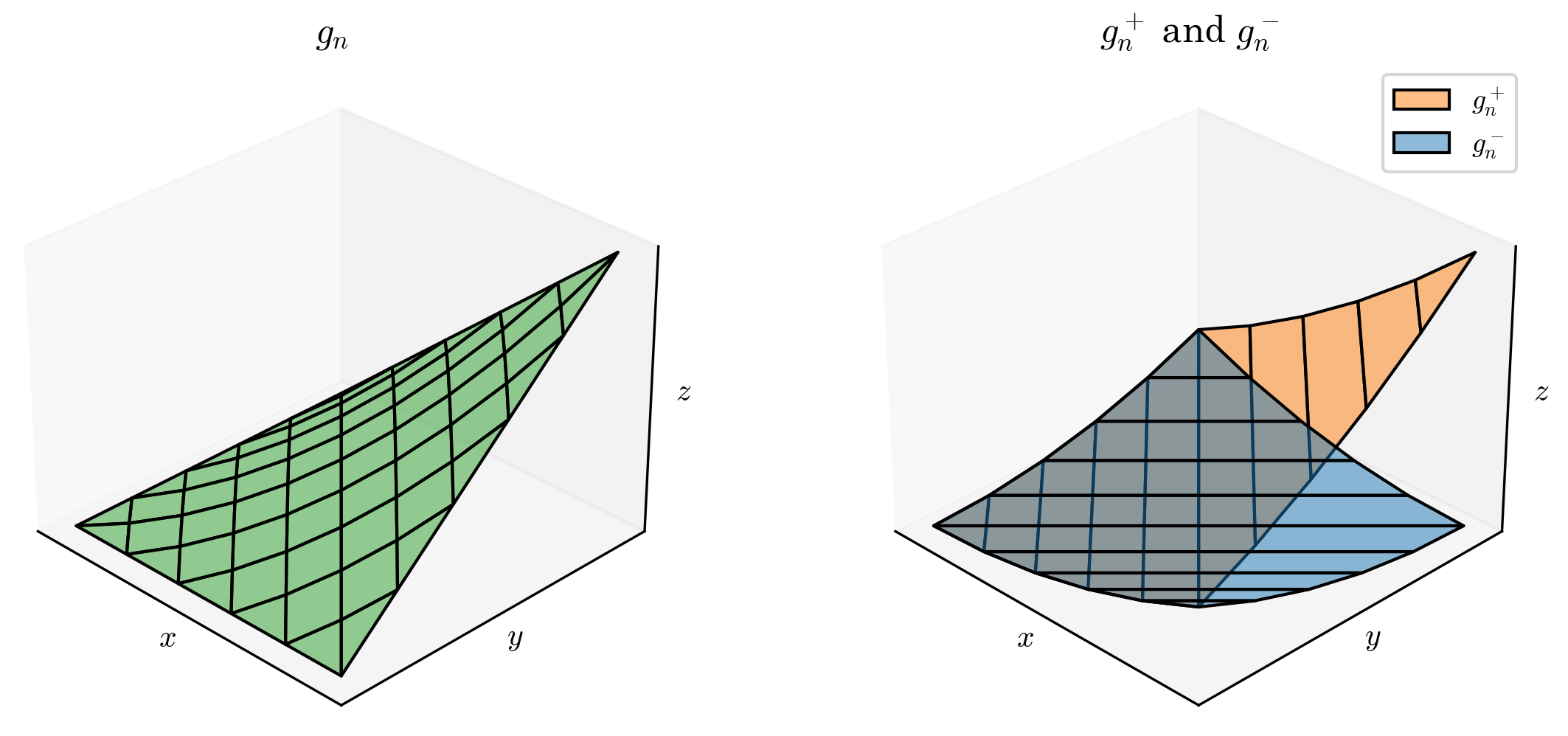}
    \caption[Illustration of $g_n$, $g_n^+$, and $g_n^-$ ($n=5$)]{Illustration of $g_n$, $g_n^+$, and $g_n^-$ ($n=5$)}
\label{fig:3D_plot}
\end{figure}

\subsection{Approximation error}

\begin{lemma}
    \label{max_R0}
    Let $\rho>0$. Let $R_{\rho}$ be the closed square domain defined by $\{(x,y) \in \mathbb{R}^2: |x|+|y| \leq \rho\}$. Let $E_{\rho}: R_{\rho} \to \mathbb{R}, \quad E_{\rho}(x,y) = xy$. Then, the maximum value of $|E_{\rho}|$ is $\rho^2/4$.
\end{lemma}

\begin{proof}
    The region $R_{\rho}$ is defined by:
\begin{align*}
    x - \rho &\leq y \leq x + \rho \\
    - x - \rho &\leq y \leq  - x + \rho
\end{align*}
Figure \ref{fig:R_0_2} provides an illustration of the region $R_{\rho}$ and the value of $|E_{\rho}|$.

The function $E_{\rho}$ has no local extrema in the interior of $R_{\rho}$ because:
\[
det(H_{E_{\rho}}) = 
\begin{vmatrix}
\frac{\partial^2 E_{\rho}}{\partial x^2} & \frac{\partial^2 E_{\rho}}{\partial x \partial y} \\
\frac{\partial^2 E_{\rho}}{\partial y \partial x} & \frac{\partial^2 E_{\rho}}{\partial y^2}
\end{vmatrix}
=
\begin{vmatrix}
0 & 1 \\
1 & 0
\end{vmatrix} = -1 <0
\]

Therefore, $E_{\rho}$ reaches its extrema on  $\partial R_{\rho}$, the boundary of $R_{\rho}$, which is composed of the four edges of the square $R_{\rho}$. Because  $|E_{\rho}(x,y)|$ is symmetrical with respect to the lines $x=0$ and $y=0$, the domain on which $\max |E_{\rho}(x,y)|$ is evaluated can be reduced to one edge of the square. It follows that:
\begin{align*}
&\max_{(x,y) \in R_{\rho}}{|E_{\rho}(x,y)|} = \max_{(x,y) \in \partial R_{\rho}}{|E_{\rho}(x,y)|} = \max_{\substack{
  y = x+ \rho \\
  - \rho \leq x \leq 0 
}}{|E_{\rho}(x,y)|} 
\\ &= \max_{-\rho \leq x \leq 0}{|E_{\rho}(x,x + \rho)|} 
 = \max_{-\rho \leq x \leq 0}{|x||x+\rho|} 
= \max_{-\rho \leq x \leq 0}{-x(x+\rho)}
\\ &= \max_{|v| \leq \rho/2}{\left(\frac{\rho}{2}-v\right)\left(\frac{\rho}{2}+v\right)}
 = \max_{|v| \leq \rho/2}{\left(\frac{\rho^2}{4} - v^2\right)}
= \frac{\rho^2}{4} 
\end{align*}
\qed
\end{proof}

As illustrated in Figure \ref{fig:R_0_2}, $|E_{\rho}|$ reaches its maximum value at the center of each edge of $R_{\rho}$.

\begin{figure}[ht]
    \centering
    \includegraphics[width=1.0\linewidth]{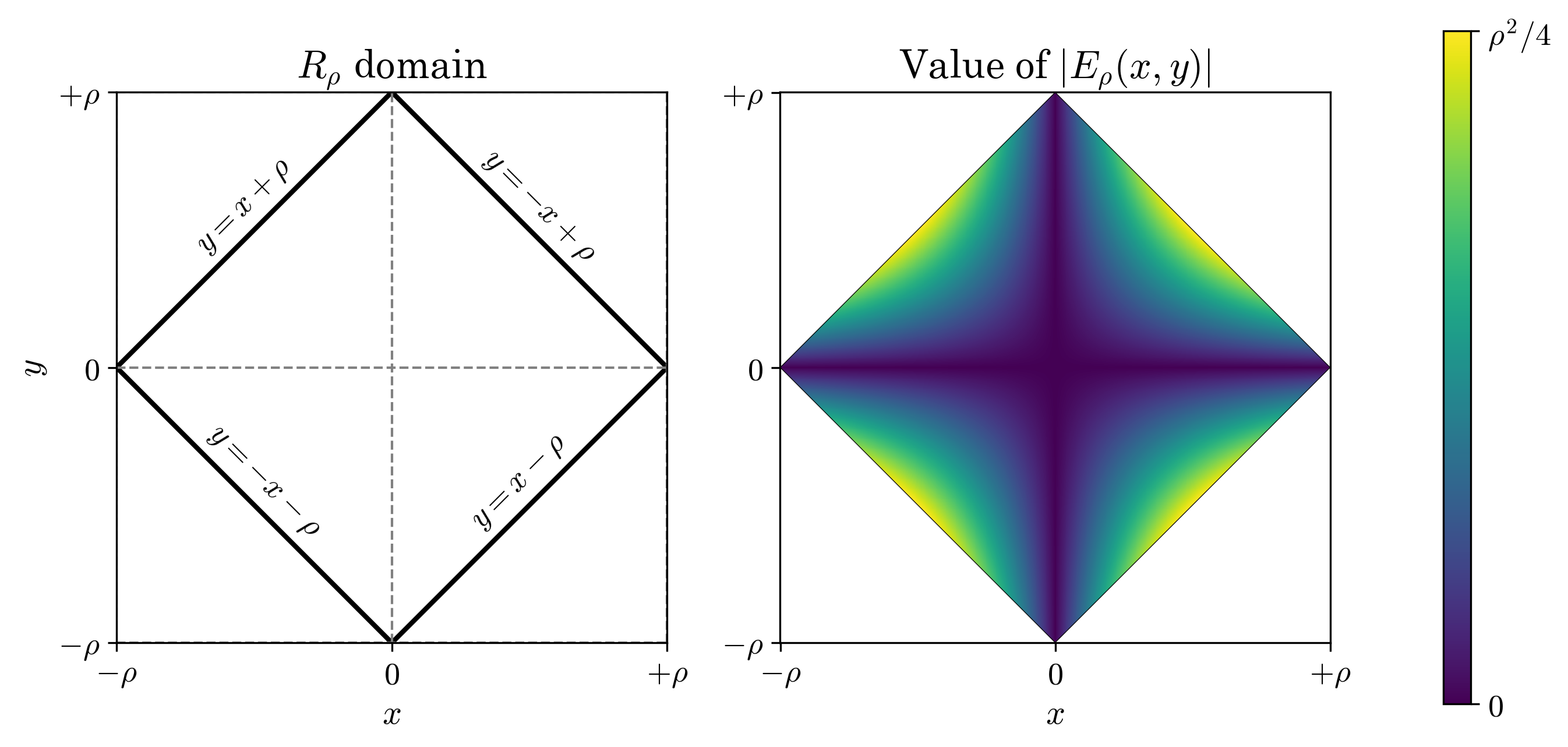}
    \caption[Illustration of $R_{\rho}$ and $|_{\rho}E|$]{Illustration of $R_{\rho}$ and $|E_{\rho}|$}
\label{fig:R_0_2}
\end{figure}

\begin{lemma}
\label{mean_R0}
    The average value of $|E_{\rho}|$ is $\rho^2/12$.
\end{lemma}

\begin{proof}
The average value of $|E_{\rho}|$ is defined as the double integral of $|E_{\rho}|$ over $R_{\rho}$, divided by the area of $R_{\rho}$. $R_{\rho}$ is a square with sides of length $\rho \sqrt{2}$. Thus, the area of $R_{\rho}$ is $2 \rho^2$. In addition, since the function $|E_{\rho}|$ is symmetrical with respect to the lines $x=0$ and $y=0$, the double integral of $|E_{\rho}|$ over $R_{\rho}$ is equal to four times the double integral of $|E_{\rho}|$ over the region of $R_{\rho}$ where $x \leq 0$ and $y \geq 0$. It follows that:
\begin{align*}
\overline{|E_{\rho}|} &= \frac{1}{|R_{\rho}|} \iint_{R_{\rho}} |E_{\rho}(x,y)| \,dx\,dy =
\frac{4}{2 \rho^2} \int_{x=-\rho}^{0}  \int_{y=0}^{x+\rho} |E_{\rho}(x,y)| \,dx\,dy 
\\ &= \frac{2}{\rho^2} \int_{x=-\rho}^{0}  \int_{y=0}^{x+\rho} |x| |y|\,dx\,dy
 = \frac{2}{\rho^2} \int_{x=-\rho}^{0}  (-x) \left( \int_{y=0}^{x+\rho} y\,dy \right)\,dx 
\\ &=  \frac{2}{\rho^2} \int_{x=-\rho}^{0}  \frac{1}{2}(- x) (x+\rho)^2\,dx 
= \frac{1}{\rho^2} \int_{v=0}^{\rho} (\rho-v)v^2\,dv
\\ &= \frac{1}{\rho} \int_{v=0}^{\rho}{v^2 \,dv} - \frac{1}{\rho^2} \int_{v=0}^{\rho}{v^3\,dv} = \frac{\rho^2}{3} - \frac{\rho^2}{4} = \frac{\rho^2}{12}
\end{align*}
\qed
\end{proof}

\begin{theorem}
\label{theorem_cpwl}
The \gls{cpwl} function $g_n$ approximates $f$ with a maximum error of $1/(16n^2)$ and an average error of $1/(48n^2)$.
\end{theorem}

\begin{proof}
Let $E : [0,1]^2 \to \mathbb{R}, \enspace E(x, y) = f(x,y) - g_n(x,y)$ be the function measuring the difference between $f$ and $g_n$.

Let $(j,k)$ be such that \(int(D_{n,j,k}) \neq \varnothing\). Let $(x_0,y_0) = (\frac{j-k}{2n}+\frac{1}{2},\frac{j+k+1}{2n}-\frac{1}{2})$.

First, note that $D_{n,j,k}$ is the square domain defined by $|x-x_0|+|y-y_0| \leq \frac{1}{2n}$ (Figure \ref{fig:domain}). In addition, for all $(x,y) \in D_{n,j,k}$, we have:
\begin{align*}
&g_n(x,y) =
\\ &\left( \frac{j+k+1-n}{2n} \right) x + \left( \frac{j-k+n}{2n} \right) y  -
\left( \frac{j+k+1-n}{2n} \right)\left( \frac{j-k+n}{2n} \right)
\\ &= y_0  x + x_0 y - x_0 y_0
\end{align*}

Therefore,
\[ \forall (x,y) \in D_{n,j,k}, \quad E(x,y) = xy - y_0 x - x_0 y + x_0 y_0 = (x-x_0)(y-y_0) \]

Let $\rho = \frac{1}{2n}$. Let $(u,v) = (x-x_0,y-y_0)$. We have $(x,y) \in D_{n,j,k} \Leftrightarrow (u,v) \in R_{\rho}$ and $E(x,y) = E_{\rho}(u,v)$.

Consequently, according to lemma 1: 
\[\max_{(x,y) \in D_{n,j,k}}{|E(x,y)|} = \max_{(u,v) \in R_{\rho}}{|E_{\rho}(u,v)|} = \frac{\rho^2}{4} = \frac{1}{16n^2}\]

According to lemma 2:
\[\frac{1}{|D_{n,j,k}|} \int_{D_{n,j,k}} |E| = \frac{1}{|R_{\rho}|} \int_{R_{\rho}} |E_{\rho}| = \frac{\rho^2}{12} = \frac{1}{48n^2}\]

Note that this result is independent of the value of $(j,k)$. As a result, the maximum and average value of $|E|$ over $[0,1]^2$ is $\frac{1}{16n^2}$ and $\frac{1}{48n^2}$, respectively. 
\qed
\end{proof}

\section{\gls{milp} formulation of the bilinear term}
While the results presented in the previous section apply to the product of two variables bounded between 0 and 1, in this section we show that they can be extended to the product of variables with any finite bounds. We first present the \gls{milp} formulation of $g_n$, which we refer to as the ``normalized case," and then establish its extension to the general case with arbitrary finite bounds.

\label{sec:milp_formulation}
\subsection{Normalized case}
Because $g_n$ is a \gls{cpwl} function, the equation $z = g_n(x,y)$ can be modeled in a \gls{milp} problem. A common and strong \gls{milp} formulation of a \gls{cpwl} function is the \gls{mc} formulation described below \cite{vielma_mixed_2010}.
\begin{align}
\label{eq:z_comp}
z &= \sum_{p=1}^{P} {z_p}\\
\label{eq:xy_comp}
\begin{bmatrix} x \\ y \end{bmatrix} &= \sum_{p=1}^{P} {\begin{bmatrix} x_p \\ y_p \end{bmatrix}}\\
\label{eq:z_eq}
z_p &= \bm{m}_p^T\begin{bmatrix} x_p \\ y_p \end{bmatrix} + \delta_p c_p, &p \in \{1, ..., P\}\\
\label{eq:domain_p}
\bm{A}_p \begin{bmatrix} x_p \\ y_p \end{bmatrix} &\leq \delta_p \bm{b}_p, &p \in \{1, ..., P\}\\
\label{eq:one_piece}
\sum_{j=1}^{P} {\delta_j} &=1\\
x, y, z &\in [0,1]\\
x_p, y_p, z_p &\in [0,1], &p \in \{1, ..., P\}\\
\delta_p &\in \{0,1\}, &p \in \{1, ..., P\}
\end{align}

Here, $P$ is the number of linear pieces of the \gls{cpwl} function. The variables $x$, $y$, and $z$ are expressed as the sum of their $P$ components $x_p$, $y_p$, and $z_p$, respectively (Equations~(\ref{eq:z_comp}) and (\ref{eq:xy_comp})). These components are equal to zero when the piece is deactivated, \textit{i.e.}, $\delta_p = 0$, and to $x$, $y$, and $z$ when the piece is activated, \textit{i.e.}, $\delta_p = 1$ (Equations (\ref{eq:z_eq}), (\ref{eq:domain_p}), (\ref{eq:one_piece})). $\bm{A}_p \begin{bmatrix} x \\ y \end{bmatrix} \leq \bm{b}_p$ represent the set of linear inequalities defining the domain of the linear piece $p$, and $z = \bm{m}_p^T \begin{bmatrix} x \\ y \end{bmatrix}+c_p$ represent its affine equation. 

For a bivariate \gls{cpwl} function, the linear inequalities defining each piece’s domain correspond to the edges of its polygonal domain. Specifically, for the ``Triangle'' representation of $g_n$, the domain of each piece is defined by a set of 3 linear inequalities, whereas for the ``Square'' representation of $g_n$, the domain $D_{n,j,k}$ of each piece is defined by a set of 4 linear inequalities.

For the ``\gls{dc}'' representation, we add equation (\ref{eq:dc}), where each convex component $z^+$ and $z^-$ is formulated by its own set of \gls{mc} constraints. The domain of each piece of $z^+$ and $z^-$ is defined by a set of 2 linear inequalities defining the boundaries of the corresponding strip.
\begin{equation}
\label{eq:dc}
z = z^+ - z^-
\end{equation}

The detailed \gls{milp} formulation of $z=g_n(x)$ under the ``Square'' and ``\gls{dc}'' representations can be found in Online Resource 1 (ESM\_1).

\subsection{Extension to the case with arbitrary bounds}
\label{sec:general_case}

The results above can be generalized to the product of two variables with arbitrary bounds  $X \in [\underline{X},\overline{X}]$ and $Y \in [\underline{Y},\overline{Y}]$ by including the three linear equations below.

\begin{align}
X &= \underline{X} + (\overline{X} - \underline{X})x
\\ Y &= \underline{Y} + (\overline{Y} - \underline{Y})y
\\ Z &= \underline{X}\underline{Y} + \underline{Y}(\overline{X}-\underline{X}) x + \underline{X}(\overline{Y}-\underline{Y}) y  + (\overline{Y}-\underline{Y})(\overline{X}-\underline{X}) z
\end{align}

\begin{theorem}
\label{th:approx_error}
The average and maximum approximation error of the variable $Z$ with respect to the bilinear term $XY$ are $(\overline{X} - \underline{X})(\overline{Y} - \underline{Y}) / (16n^2)$ and $(\overline{X} - \underline{X})(\overline{Y} - \underline{Y}) / (48n^2)$
\end{theorem}
\begin{proof}
The proof stems from rescaling $X$ and $Y$ and applying Theorem \ref{theorem_cpwl}.
Let $G_n$ the function defined on $[\underline{X},\overline{X}] \times [\underline{Y},\overline{Y}]$ and expressed by:

\begin{align*}
G_n(X,Y) &= \underline{X}\underline{Y} + \underline{Y}(X-\underline{X})+ \underline{X}(Y-\underline{Y}) \\ &+ (\overline{Y}-\underline{Y})(\overline{X}-\underline{X}) \cdot g_n \left( \frac{X-\underline{X}}{\overline{X}-\underline{X}}, \frac{Y-\underline{Y}}{\overline{Y}-\underline{Y}} \right)
\end{align*}

$G_n$ is a \gls{cpwl} function because $g_n$ is a \gls{cpwl} function, and is composed of the same number of linear pieces as $g_n$.


Let $(X,Y) \in [\underline{X},\overline{X}] \times [\underline{Y},\overline{Y}]$. Let $(x,y) = \left( \frac{X-\underline{X}}{\overline{X}-\underline{X}}, \frac{Y-\underline{Y}}{\overline{Y}-\underline{Y}} \right)$. We have:
\begin{align*}
&XY - G_n(X,Y) 
\\ &= (\underline{X} + x (\overline{X}-\underline{X}))(\underline{Y} + y (\overline{Y}-\underline{Y})) - G_n(X,Y) 
\\ &= \underline{X}\underline{Y} + \underline{Y}(\overline{X}-\underline{X})x + \underline{X}(\overline{Y}-\underline{Y})y + (\overline{Y}-\underline{Y})(\overline{X}-\underline{X}) xy - G_n(X,Y) 
\\ &= (\overline{Y}-\underline{Y})(\overline{X}-\underline{X})xy -  (\overline{Y}-\underline{Y})(\overline{X}-\underline{X}) g_n(x,y)
\\ &= (\overline{Y}-\underline{Y})(\overline{X}-\underline{X}) (xy - g_n(x,y))
\end{align*}

Therefore, according to Theorem \ref{theorem_cpwl}, the average and maximum value of $|XY - G_n(X,Y)|$ is $(\overline{X} - \underline{X})(\overline{Y} - \underline{Y}) / (16n^2)$ and $(\overline{X} - \underline{X})(\overline{Y} - \underline{Y}) / (48n^2)$, respectively. 
\qed
\end{proof}

\begin{remark}
\label{rem:finite_bounds}
This result highlights the need for the variables $X$ and $Y$ to have finite bounds. Without finite bounds, no finite approximation error can be derived.
\end{remark}

\subsection{Extension to indefinite quadratic functions}
\label{sec:indefinite_case}

More generally, the results above can be extended to any indefinite quadratic surfaces of the form 
\[
Z = q(X,Y)=aX^2+bXY+cY^2+dX+eY+f
\]
verifying 
\[
\det(\nabla^2 q) = \det\begin{pmatrix}
2a & b \\
b & 2c
\end{pmatrix} = 4ac-b^2 < 0
\]

Indeed, if $4ac-b^2 < 0$, there is an invertible transformation matrix $\bm{B}$ such that $(x,y) \in [0,1]^2$, $z = xy$, and
\begin{equation}
\label{eq:transf_matrix}
\begin{bmatrix}
X \\ Y \\Z \\ 1
\end{bmatrix} = \bm{B}
\begin{bmatrix}
x \\ y \\z \\ 1
\end{bmatrix}
\end{equation}

Consequently, a \gls{cpwl} approximation of $Z = q(X,Y)$ can be incorporated into the \gls{milp} formulation by adding equation~(\ref{eq:transf_matrix}). However, for the sake of simplicity, this extension is beyond the scope of the present manuscript and is therefore not considered in the rest of the article.

\subsection{Formulation size}

A size comparison of the three \gls{milp} formulations is provided in Table~\ref{tab:size_summary}. For all formulations, moving from the normalized case to the general case introduces only three additional continuous variables and three additional linear constraints. 

\begin{table}[htbp]
\centering
\caption{Size comparison of the \gls{milp} formulations of $z=g_n(x,y)$}
\label{tab:size_summary}
\begin{tabular}{p{2.9cm} p{2.1cm} p{2.9cm} p{2.1cm}}
\toprule
\textbf{\gls{milp} formulation} & \textbf{\# of binary variables} & \textbf{\# of continuous variables} & \textbf{\# of linear constraints} \\
\midrule
Triangle & $4n^2$ & $12n^2+1$ & $16n^2+4$ \\
Square & $2n(n+1)$ & $6n(n+1)+1$ & $10n(n+1)+4$ \\
DC & $4n$ & $12n+3$ & $14n+9$ \\
\midrule
General $X$ and $Y$ & +0 & +3 & +3 \\
\bottomrule
\end{tabular}
\end{table}

For the rest of the article we consider that the equation $Z=G_n(X,Y)$ represents the set of variables and linear equations (1)-(12) that defines $Z$ as the \gls{cpwl} approximation of $XY$.

\section{Sequentially Coupled Bilinear Programs}
\label{sec:ncqp}

This section introduces a specific class of optimization problems for which the \gls{cpwl}/\gls{milp} approximation formulated above proves to be computationally beneficial.
We call this class of optimization problems: ``Sequentially Coupled Bilinear Programs'', and we will refer to them as $SCBP$.

\subsection{General problem formulation}
The general formulation of the \gls{scbp} problem is described below.
\begin{align}
\label{eq:ob_qp}
\max \enspace & \sum_{t=1}^T w_t X_t Y_t \\
\label{eq:storage_balance}
\text{s.t.} \quad & S_t - S_{t-1} = I_t - X_t, \quad &t = 1,\dots,T \\
\label{eq:storage_elevation}
& Y_t = h(S_t), \quad &t = 1,\dots,T \\
\label{eq:X_bound_t}
& X_t \in [\underline{X}_t,\overline{X}_t], \quad &t = 1,\dots,T \\
\label{eq:Y_bound_t}
& Y_t \in [\underline{Y}_t,\overline{Y}_t], \quad &t = 1,\dots,T
\end{align}
where $S_0$, $(I_t)_{t=1,\dots,T}$, and $(w_t)_{t=1,\dots,T}$ are static parameters, $h$ is a linear or \gls{cpwl} function, and $(X_t)_{t=1,\dots,T}$, $(Y_t)_{t=1,\dots,T}$, and $(S_t)_{t=1,\dots,T}$ are continuous decision variables.

More specifically, this class of problems is characterized by the following criteria:
\begin{itemize}
    \item The objective function (\ref{eq:ob_qp}) is a sum of bilinear terms.
    \item The variables that compose the bilinear terms have finite bounds (\ref{eq:X_bound_t}), (\ref{eq:Y_bound_t}).
    \item The constraints of the problem are linear.
    \item The problem includes a sequential coupling constraint (\ref{eq:storage_balance}) that governs the state evolution over time.
    \item For each bilinear term, one variable ($Y_t$) is a function of the system state, while the other ($X_t$) dictates the change in that state.
\end{itemize}

A concrete example of this class of problems is optimal hydropower scheduling. In this context, the variables $X_t$, $Y_t$, and $S_t$ represent the water release, the hydraulic head, and the reservoir storage, respectively. The parameters $w_t$ and $I_t$ represent the market value of hydropower production and the exogenous water inflow into the reservoir. 
Another relevant application is dynamic pricing in inventory revenue management. Here, the variables $X_t$, $Y_t$, and $S_t$ represent the sales volume, the unit selling price, and the remaining inventory level, respectively. The price $Y_t$ is dynamically adjusted based on the available inventory $h(S_t)$ to capture scarcity pricing, while the parameters $w_t$ and $I_t$ represent the time-discount factor for revenue and the scheduled inventory restock quantities.

For sake of simplicity, we assume for the rest of the article that equation (\ref{eq:storage_elevation}) is a linear constraint. In that case, $SCBP$ is composed of $3T$ continuous variables, $2T$ linear constraints, and has an objective function composed of $T$ bilinear terms.

\subsection{\gls{cpwl} approximation of the problem}

The \gls{scbp} problem described above can be converted to a \gls{milp} problem by applying the \gls{cpwl} approximation described in Section \ref{sec:general_case} to each bilinear term of the objective function.

Concretely, the objective function must be replaced by equation (\ref{eq:obj_pwl}). For the rest of the article, we will refer to the system of equation (\ref{eq:obj_pwl}), (\ref{eq:storage_balance})-(\ref{eq:Y_bound_t}) as $MILP(n)$.
\begin{align}
\label{eq:obj_pwl}
\max \quad & \sum_{t=1}^T {w_t G_n(X_t, Y_t)}
\end{align}

Technically speaking, equation (\ref{eq:obj_pwl}) is actually composed of the linear objective function $\sum_{t=1}^T {w_t Z_t}$ and the linear constraints linking $X_t$, $Y_t$, $Z_t$ described in Section \ref{sec:milp_formulation}.
The number of binary variables, continuous variables, and linear constraints of $MILP(n)$ depends on the \gls{milp} formulation of $Z_t=G_n(X_t,Y_t)$ and is summarized in Table \ref{tab:size_MILP_n}.

\begin{table}[htbp]
\centering
\caption{Size comparison of $SCBP$ and $MILP(n)$}
\label{tab:size_MILP_n}
\begin{tabular}{p{2.8cm} p{2.0cm} p{2.6cm} p{2.3cm}}
\toprule
\textbf{Problem formulation} & \textbf{\# of binary variables} & \textbf{\# of continuous variables} & \textbf{\# of linear constraints} \\
\midrule
$SCBP$ & $0$ & $3T$ & $2T$ \\
\midrule
$MILP(n)$, Triangle & $4n^2T$ & $(12n^2+7)T$ & $(16n^2+9)T$ \\
$MILP(n)$, Square & $2n(n+1)T$ & $(6n(n+1)+7)T$ & $(10n(n+1)+9)T$ \\
$MILP(n)$, DC & $4nT$ & $(12n+9)T$ & $(14n+14)T$ \\
\bottomrule
\end{tabular}
\end{table}

\subsection{Feasibility and quality of the \gls{cpwl} solution}

Because $SCBP$ and $MILP(n)$ share the same feasible region, a feasible solution of $MILP(n)$ is necessarily a feasible solution of $SCBP$. However, a feasible solution may not have the same objective value in $MILP(n)$ and $SCBP$.

Let $\bm{S} = (X_t,Y_t,S_t)_{t=1,...,T}$ be a feasible solution. Let $F_{MILP(n)}(\bm{S})$ and $F_{SCBP}(\bm{S})$ be the objective function value of the solution $\bm{S}$ for $MILP(n)$ and $SCBP$.

\begin{theorem}
\label{th:epsilon_n}
     For a given feasible solution $\bm{S}$, the maximum difference $\varepsilon_n$ between $F_{MILP(n)}(\bm{S})$ and $F_{SCBP}(\bm{S})$ is:
     \[
     \varepsilon_n = \frac{1}{16n^2}\sum_{t=1}^T |w_t| |\overline{X_t}-\underline{X_t}||\overline{Y_t}-\underline{Y_t}|
     \]
\end{theorem}

\begin{proof}
 \begin{align*}
&|F_{MILP(n)}(\bm{S}) - F_{SCBP}(\bm{S})| \\ &= \left| \sum_{t=1}^T {w_t G_n(X_t, Y_t)} -  \sum_{t=1}^T {w_t X_t Y_t}\right| 
 = \left| \sum_{t=1}^T {w_t (G_n(X_t, Y_t) - X_t Y_t)}\right| 
 \\ &\leq \sum_{t=1}^T |w_t| |G_n(X_t, Y_t) - X_t Y_t| \leq \frac{1}{16n^2}\sum_{t=1}^T |w_t| |\overline{X_t}-\underline{X_t}||\overline{Y_t}-\underline{Y_t}|
\end{align*} 
\qed
\end{proof}

\begin{remark}
    Note that $\varepsilon_n$ only depends on the input parameters of the problem and not on the solution $\bm{S}$. In addition, the difference between $F_{MILP(n)}(\bm{S})$ and $F_{SCBP}(\bm{S})$ can be arbitrarily close to zero by increasing the degree of accuracy $n$.
\end{remark}

\begin{theorem}
\label{th:epsilon_n_gap}
If a \gls{milp} solver identifies a feasible solution $\bm{S}$ with an absolute optimality gap $Gap_{MILP(n)}^{abs}$, then the gap $Gap_{SCBP}^{abs}$ in the objective value between the identified solution and the optimal $SCBP$ solution verifies:
\[
Gap_{SCBP}^{abs} \leq Gap_{MILP(n)}^{abs} + 2\varepsilon_n
\]
\end{theorem}

\begin{proof}
    Let $\bm{S}$ be a feasible solution. Let $F_{MILP(n)}^{UB}$ be an upper bound of the objective function value of $MILP(n)$ and $Gap_{MILP(n)}^{abs} = F_{MILP(n)}^{UB} - F_{MILP(n)}(\bm{S})$.

    According to Theorem \ref{th:epsilon_n},
    \[
    F_{SCBP}(\bm{S}) \geq F_{MILP(n)}(\bm{S}) - \varepsilon_n 
    \]

    In addition, for any feasible solution $\bm{S'}$, 
    \[
    F_{SCBP}(\bm{S'}) \leq F_{MILP(n)}(\bm{S'}) + \varepsilon_n \leq  F_{MILP(n)}^{UB} + \varepsilon_n
    \]
    Therefore, 
    \[
    F_{SCBP}^{UB} \leq  F_{MILP(n)}^{UB} + \varepsilon_n
    \]

    Consequently,
    \begin{align*}
    Gap_{SCBP}^{abs} &= F_{SCBP}^{UB} - F_{SCBP}(\bm{S}) 
    \\ &\leq F_{MILP(n)}^{UB} + \varepsilon_n - F_{MILP(n)}(\bm{S}) + \varepsilon_n 
    \\ &= Gap_{MILP(n)}^{abs} + 2\varepsilon_n
    \end{align*}
    \qed
\end{proof}

\section{Numerical experiments}

This section illustrates the performances of the \gls{milp} formulations enabled by the \gls{cpwl} approximation of the bilinear term. In section \ref{sec:data_desc}, the problem instances, mathematical formulations, solvers, \gls{cpwl} representations and degrees of accuracy used to evaluate the computational performances are described. In section \ref{sec:exp_results}, the experiment results are compared in terms of solve status, solve time, and optimality gap. The code used for the computational experiments can be found at \cite{Ploussard2026PWLReformulations}.

The algorithms are implemented in Python (3.12.10) and the problems are modeled with OR-Tools (9.14.6206).  Models are run on an Intel 2.60-GHz machine with 14 cores and 32 GB of RAM. 

\subsection{Description of the experiments}

\label{sec:data_desc}

The \gls{milp} formulations are tested over two sets of optimization problems: 
\begin{itemize}
    \item a subset of valid QPLIB instances
    \item a set of randomly-generated \gls{scbp} instances
\end{itemize}

\subsubsection{QPLIB instances}

The QPLIB library \cite{Furini2019QPLIB} contains a collection of 453 \gls{qp} problem instances of various types, including but not limited to: \gls{qp}, QCP, MIQP, \gls{miqcp}, Nonconvex \gls{qp}. More details are provided in \cite{Furini2019QPLIB}.

We select a subset of problem instances of reasonable size that exhibit the properties targeted by our method. More specifically, we select a subset of problem instances that verify the following conditions:
\begin{enumerate}[label=\alph*)]
    \item Problems are composed of at most 200 variables and 200 constraints
    \item All quadratic terms are bilinear (only $XY$ terms and no $X^2$ terms)
    \item All quadratic terms are only composed of continuous variables
    \item All variables involved in quadratic terms have finite bounds
\end{enumerate}
Condition (a) is used to down-select problem instances of reasonable size that can be solved in a reasonable amount of time. Furthermore, the method proposed here has been specifically designed to address quadratic terms verifying conditions (b) and (c). Finally, condition (d) is necessary to apply the variable rescaling method described in Section \ref{sec:general_case}. As noted in Remark~\ref{rem:finite_bounds}, without finite bounds, no finite approximation error could be derived.

Note that condition (b) could be replaced by the weaker condition that quadratic terms can be grouped into indefinite quadratic functions $aX^2+bXY+cY^2+dX+eY+f$ as seen in Section \ref{sec:indefinite_case}. However, identifying valid indefinite quadratic functions from a linear combinations of quadratic terms ($X_iX_j$, $X_i^2$) can be challenging. For that reason, and for sake of simplicity, the stronger condition (b) is enforced in the down-selection.

The down-selection results in a subset of 15 instances of QPLIB problems.

\subsubsection{Random instances of $SCBP$ problems}

In addition to the 15 QPLIB instances, the \gls{cpwl} reformulation method proposed here is tested on various instances of the \gls{scbp} described in Section \ref{sec:ncqp}. The \gls{scbp} problems are built using randomly-generated values of the static parameters $S_0$, $(I_t)_{t=1,...,T}$, and $(w_t)_{t=1,...,T}$ and multiple sequence durations $T$. More specifically, the $SCBP$ problems are built using:
\begin{itemize}
    \item 10 random seeds to generate the parameter values
    \item 4 sequence durations ($T=24,48,96,168$)
\end{itemize}
This results in 40 instances of \gls{scbp} problem. Additionally, to ensure computational reproducibility, the randomly generated instances are initialized with predefined seeds \cite{Ploussard2026PWLReformulations}.

\subsubsection{Additional experiment dimensions}

In total, the method is tested on 55 problem instances: 15 QPLIB instances, and 40 random instances of \gls{scbp} problem. 

In addition to the problem instances, 4 extra dimensions are considered for these experiments, namely: 
\begin{itemize}
    \item the problem type: \gls{qp}, \gls{milp}
    \item the solver used: Gurobi, SCIP, HiGHS
    \item the \gls{cpwl} representation: Triangle, Square, DC
    \item the degree of accuracy $n$
\end{itemize}

For each instance, the original \gls{qp} and approximated \gls{milp} formulations are both solved. The \gls{qp} problems are solved using the Gurobi and SCIP solvers, as the HiGHS solver does not currently support Nonconvex \gls{qp} problems. The \gls{milp} problems are solved using the Gurobi, SCIP, and HiGHS solvers. Each \gls{milp} problem is solved using the 3 \gls{cpwl} representations (Triangle, Square, and DC) with 4 degrees of accuracy ($n=1,2,3,4$). This results in a total number of 2,090 numerical experiments. The numerical experiments follow the factorial design shown in Figure \ref{fig:case_studies}. 

For all experiments, the relative optimality gap is set to 1\%, and the time limit is set to 600 seconds. To minimize computational time, the 2,090 experiments are run in parallel, allowing only one thread by experiment regardless of the solver \cite{Ploussard2026PWLReformulations}.

\begin{figure}
    \centering
    \includegraphics[width=1.0\linewidth]{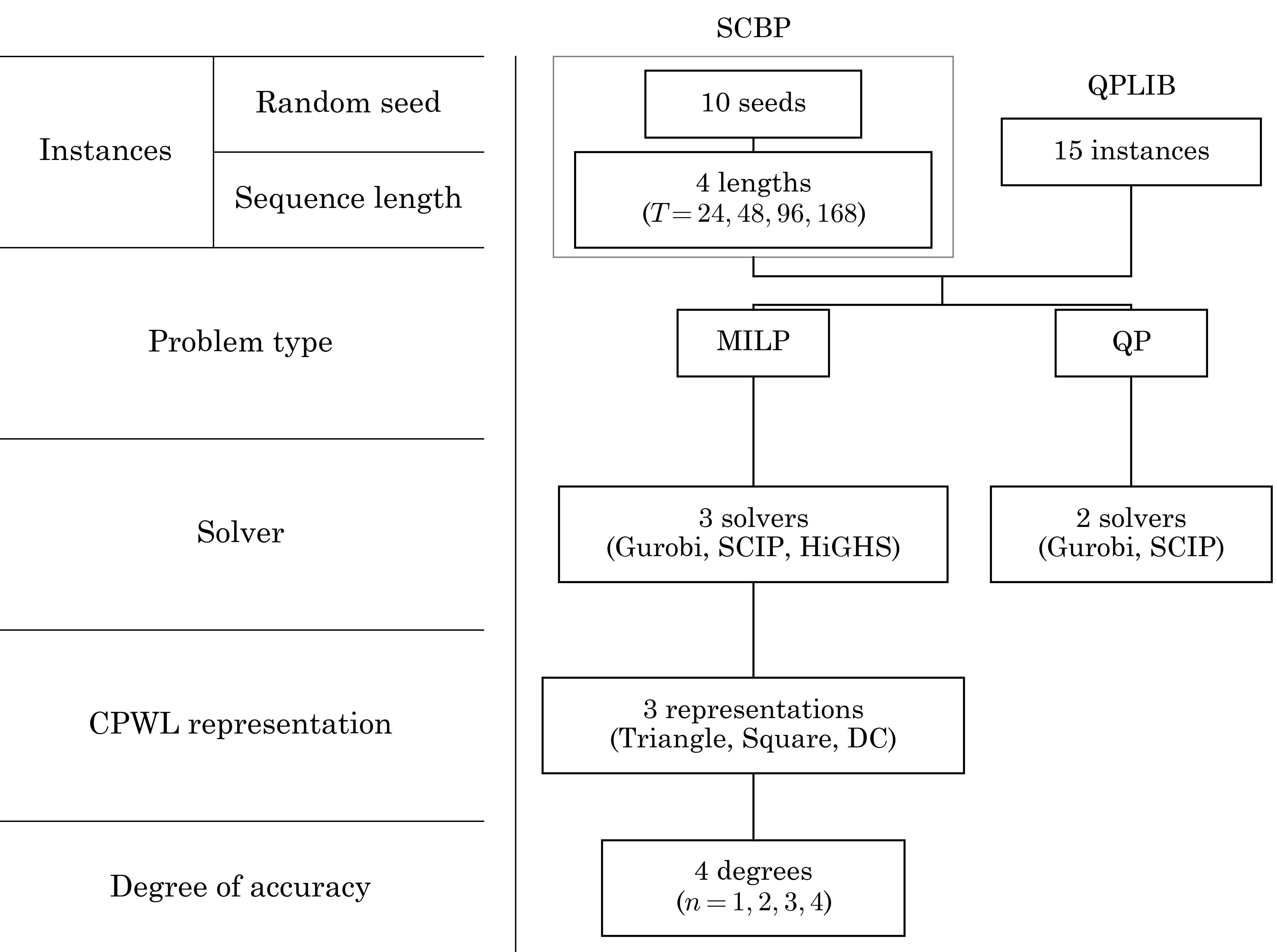}
    \caption[Dimensions and diagram of the numerical experiments]{Experimental design tree for the numerical study. Vertical connections indicate factorial combinations of dimensions, while horizontal splits indicate alternative branches whose experiment counts are summed}
\label{fig:case_studies}
\end{figure}

\subsection{Experiment results}
\label{sec:exp_results}

\subsubsection{Performances on SCBP instances}

Figure \ref{fig:NCQP_results} illustrates the average solve time for the \gls{scbp} instances for each sequence length $T$, problem type, solver, \gls{cpwl} representation, and degree of accuracy $n$. An average solve time equal to the time limit indicates that none of the experiments in the group reached the target optimality gap within the allotted time.

\begin{figure}
    \centering
    \includegraphics[width=1.0\linewidth]{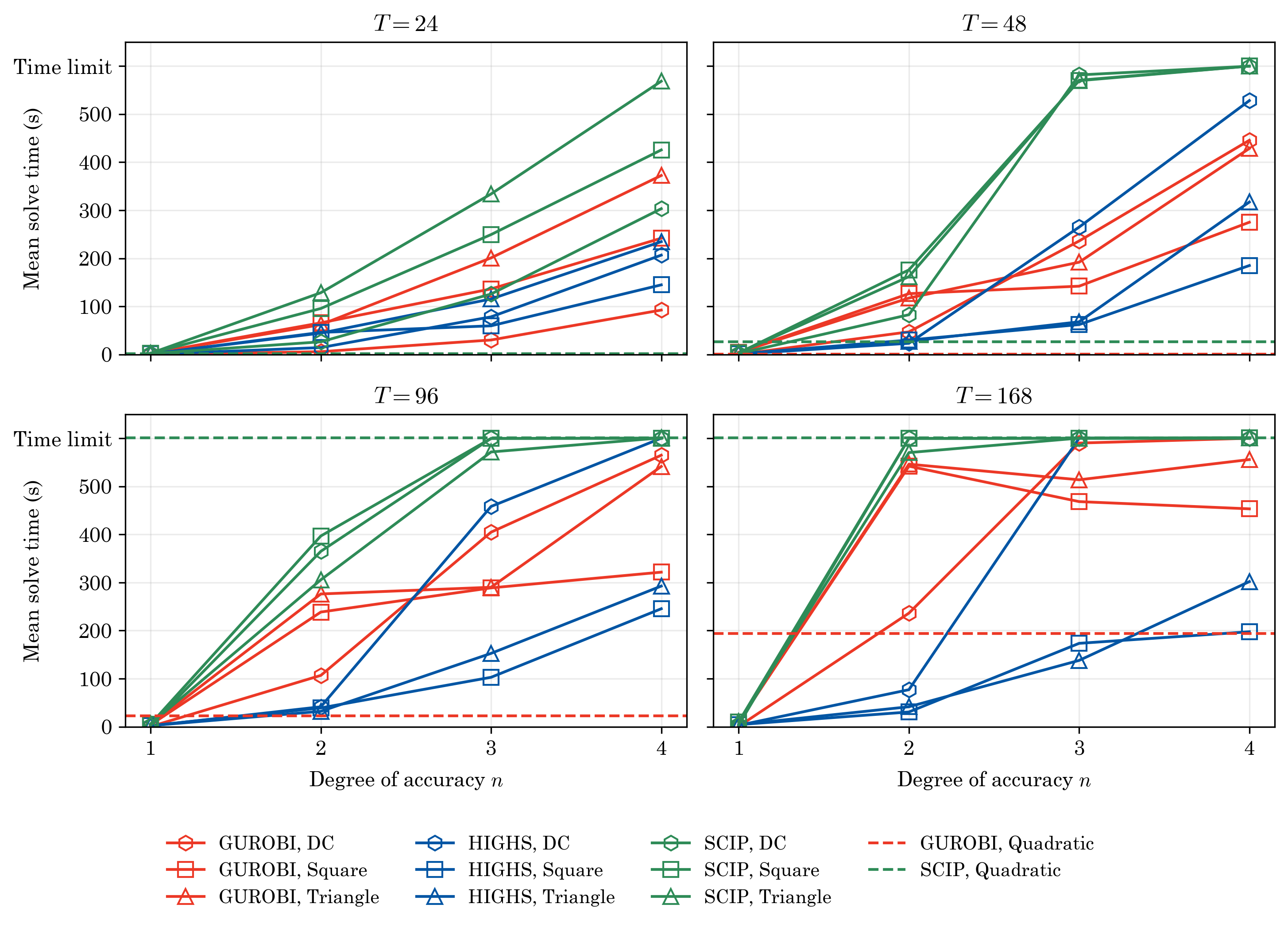}
    \caption[Average solve time of the SCBP problems]{Average solve time of the SCBP problems}
\label{fig:NCQP_results}
\end{figure}

For $T \leq 48$, both \gls{qp} solvers reach optimality within a few seconds on average, whereas the corresponding \gls{milp} formulations require longer solve times regardless of the solver used. Thus, for short sequences, the \gls{cpwl} approximation introduces unnecessary overhead and is computationally inefficient.

A different pattern emerges for longer sequences. For $T \geq 96$, SCIP is no longer able to solve the \gls{qp} instances to optimality, while HiGHS successfully solves the \gls{milp} formulations even at high degrees of accuracy. For $T=168$, the performance of HiGHS on the \gls{milp} reformulation becomes comparable to that of Gurobi on the original \gls{qp} for $n=3,4$, and noticeably better for $n=2$. This suggests that the \gls{cpwl} reformulation becomes relevant for longer sequences, where the original \gls{qp} formulation becomes more difficult for some solvers. 

Interestingly, HiGHS frequently outperforms Gurobi on the \gls{milp} formulations, underscoring the increasing maturity of open-source \gls{milp} solvers. However, since Gurobi remains faster on the original \gls{qp} formulation whenever that formulation is tractable, the \gls{cpwl} reformulation should be viewed primarily as an alternative for instances where direct \gls{qp} solution becomes inefficient or fails to reach optimality.

More generally, the strong \gls{milp} performances are largely driven by the Square representation of $g_n$. The Dolan--Moré performance profiles in Figure~\ref{fig:dolan_more} support this observation. With HiGHS, the Square representation performs best, solving more than 80\% of \gls{scbp} instances at the best observed runtime and almost all instances within a small performance ratio. The Triangle formulation is slower but remains relatively competitive, while the \gls{dc} formulation requires substantially larger performance ratios to solve comparable fractions of instances. The relatively poor performance of the \gls{dc} representation is somewhat surprising given the compactness of its \gls{milp} formulation, as reported in Table~\ref{tab:size_summary}. One possible explanation is that this compactness comes at the cost of a weaker relaxation. Specifically, because the \gls{dc} formulation represents $g_n$ as the difference of two \gls{cpwl} functions, the approximation errors and relaxation gaps associated with each component may compound in the \gls{lp} relaxation solved at each branch-and-bound node. As a result, the \gls{dc} formulation may yield weaker bounds than the Square and Triangle formulations, despite requiring fewer variables and constraints. In comparison, the superior performance of the Square formulation relative to the Triangle formulation is more intuitive. The two formulations have a similar structure, but the Square representation uses roughly half as many linear pieces. This reduction in the number of domains appears to outweigh the additional linear inequality required to describe each square domain, leading to a more efficient \gls{milp} formulation overall. Given these results, the remainder of the experimental section focuses on the Square representation. This allows the subsequent analysis to concentrate on the best-performing \gls{cpwl} representation.

\begin{figure}
    \centering
    \includegraphics[width=1.0\linewidth]{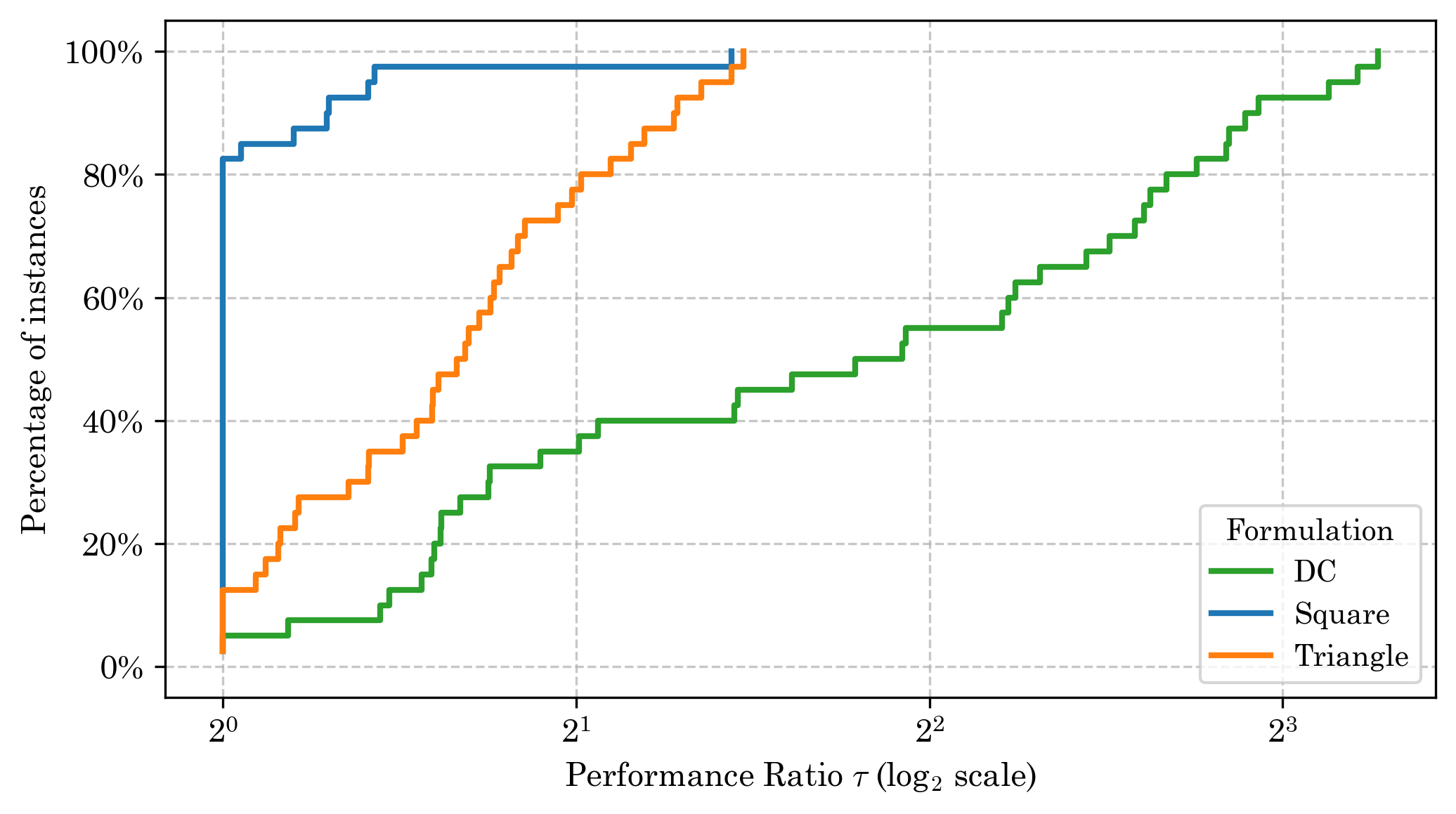}
    \caption[Dolan-Moré performance profiles of the \gls{cpwl} representations for the SCBP instances (HiGHS)]{Dolan-Moré performance profiles of the \gls{cpwl} representations for the SCBP instances (HiGHS)}
\label{fig:dolan_more}
\end{figure}

Figure~\ref{fig:scalability} reports the percentage of \gls{scbp} instances solved to optimality as a function of the sequence length. The solve rate of the open-source \gls{qp} solver SCIP falls sharply to zero for $T \geq 96$. In contrast, even with a relatively high approximation accuracy of $n=4$, the open-source \gls{milp} solver HiGHS maintains a largely stable solve rate as the sequence length increases. Notably, for $T=168$, HiGHS outperforms the commercial \gls{qp} solver Gurobi in terms of the percentage of instances solved to optimality. This highlights a practical advantage of the proposed \gls{cpwl} reformulation: by converting the original \gls{qp} into a \gls{milp}, large instances of the \gls{scbp} problem becomes accessible to mature open-source \gls{milp} solvers such as HiGHS. The results show that these solvers can provide strong computational performance on the reformulated \gls{milp} problem, while preserving, and in some cases improving, scalability for larger instances.

\begin{figure}
    \centering
    \includegraphics[width=1.0\linewidth]{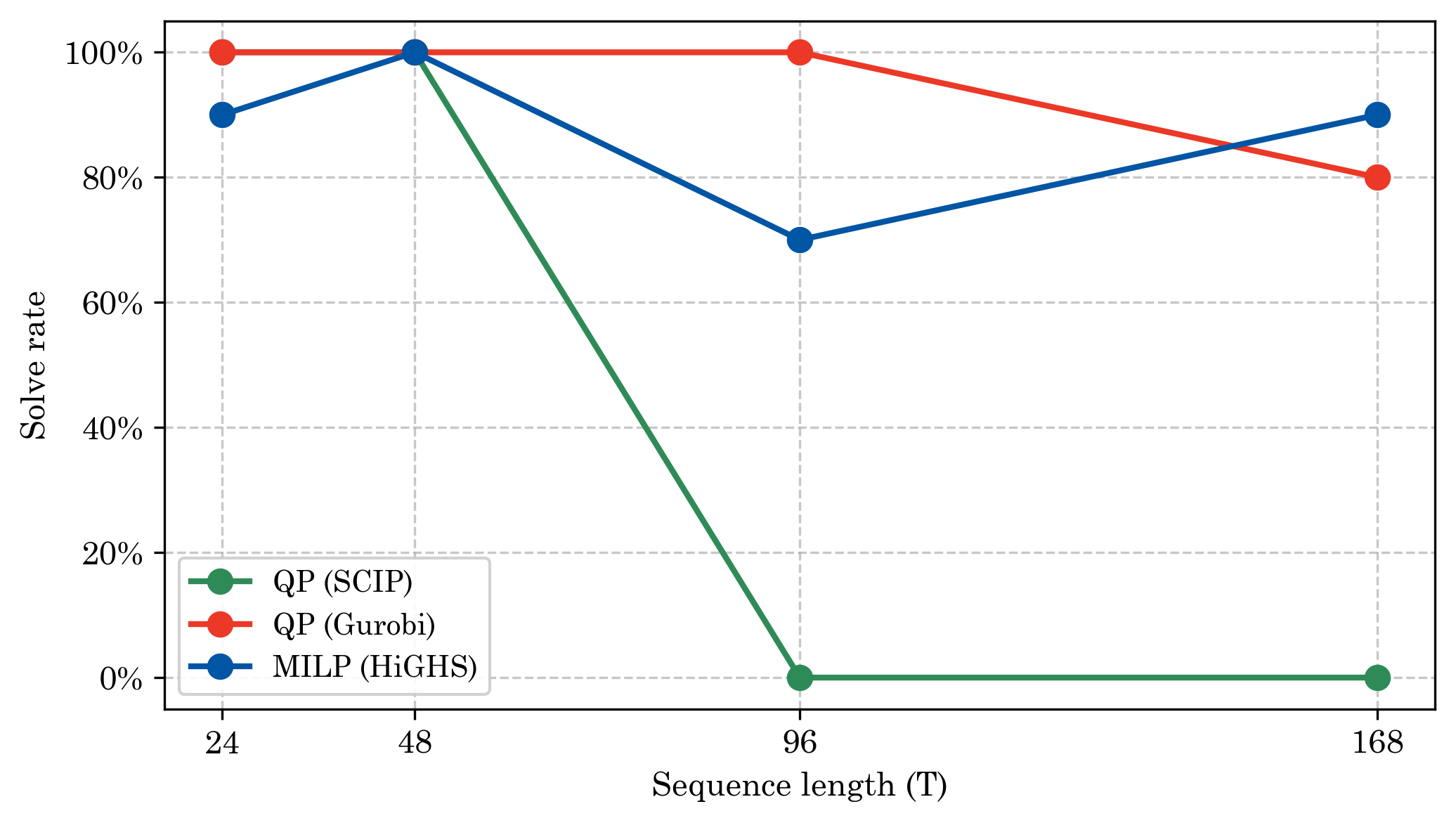}
    \caption[Solve rate of SCBP instances as a function of the sequence length ($n=4$)]{Solve rate of SCBP instances as a function of the sequence length ($n=4$)}
\label{fig:scalability}
\end{figure}

Beyond computational performance, it is important to verify that the \gls{cpwl} reformulation produces \gls{milp} solutions that remain accurate with respect to the original \gls{qp} formulation. To assess this, each \gls{milp} solution is evaluated by comparing the value of the \gls{milp} objective function to the value of the original quadratic objective function in \eqref{eq:ob_qp}. The corresponding objective \gls{cpwl} approximation gap is shown in Figure~\ref{fig:obj_approx} as a function of the degree of accuracy $n$.
The results show that the objective gap decreases rapidly as $n$ increases, consistent with the error bound established in Theorem~\ref{th:epsilon_n}. In particular, the average gap is approximately 1\% for $n=2$ and decreases to about 0.1\% for $n=4$. These results indicate that the \gls{cpwl} reformulation can achieve high objective accuracy with a relatively moderate number of linear pieces per bilinear term.

\begin{figure}
    \centering
    \includegraphics[width=1.0\linewidth]{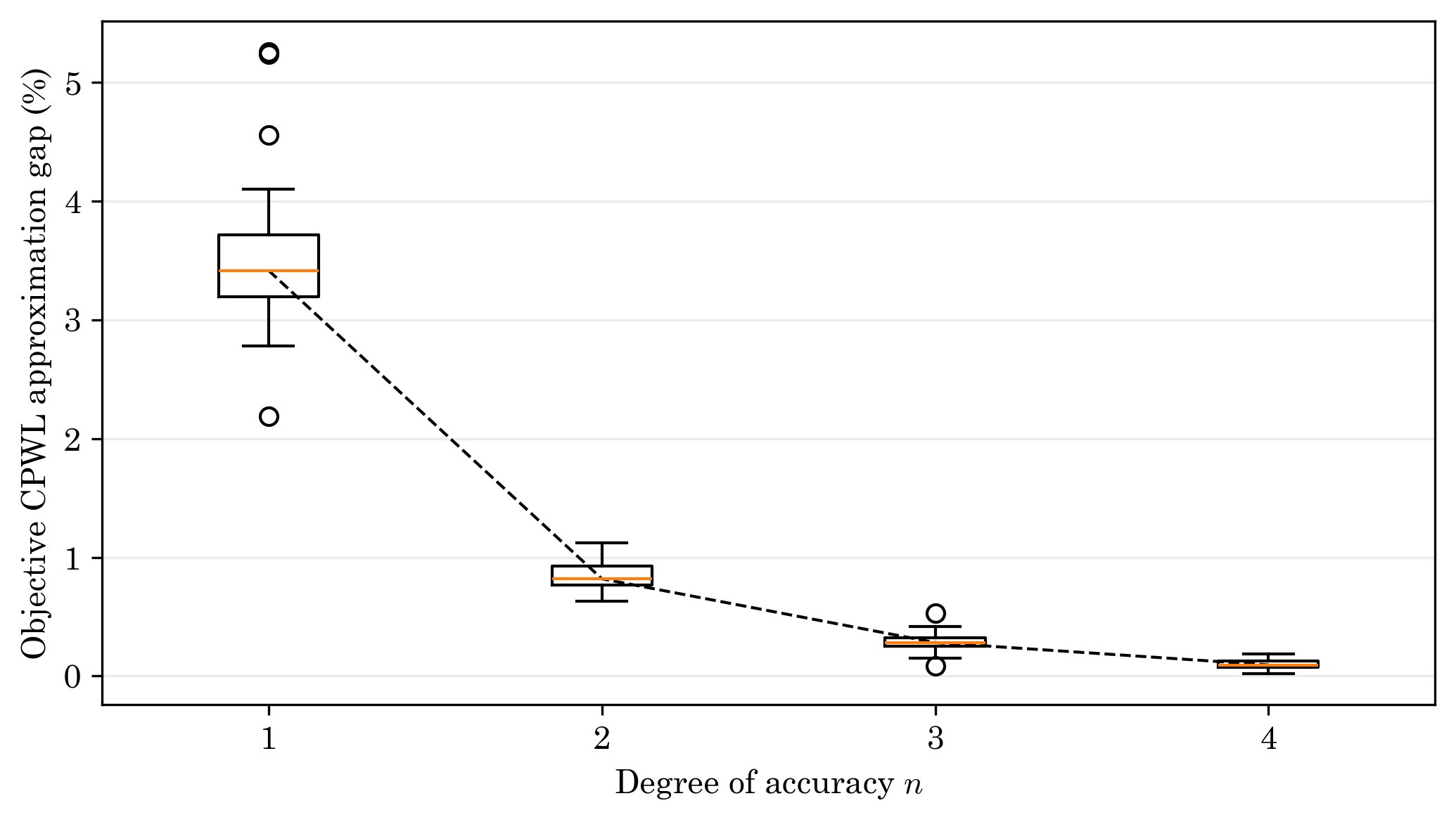}
    \caption[\gls{cpwl} approximation gap of the objective function depending on the degree of accuracy (HiGHS, Square)]{\gls{cpwl} approximation gap of the objective function depending on the degree of accuracy (HiGHS)}
\label{fig:obj_approx}
\end{figure}

\subsubsection{Performances on QPLIB instances}

The results obtained on the QPLIB instances contrast sharply with those observed on the \gls{scbp} instances. As shown in Table~\ref{tab:qplibb_results}, the \gls{milp} reformulation systematically underperforms the original \gls{qp} formulation in terms of optimality gap, with the exception of two instances: 2698 and 3385. These two cases, for which the \gls{milp} formulation achieves comparable or better performance, are highlighted in bold in Table~\ref{tab:qplibb_results}. However, even for these instances, the optimal \gls{milp} objective value differs substantially from the optimal \gls{qp} objective value, suggesting that the \gls{cpwl} approximation may not provide a faithful representation of the original \gls{qp} problem.

\begin{table}
\caption{Solver results for the \gls{qp} and \gls{milp} formulations of the QPLIB instances ($n=3$)}
\label{tab:qplibb_results}
\begin{tabular}{lrrrrrr}
\toprule
 & \multicolumn{3}{c}{\gls{qp} (SCIP)} & \multicolumn{3}{c}{\gls{milp} (HiGHS)} \\
 Instance & Primal & Dual & Gap (\%) & Primal & Dual & Gap (\%) \\
\midrule
0698 & 1927582.38 & 787778.12 & 59.13 & - & 705763.21 & - \\
1886 & -34.93 & -126.46 & 262.06 & -1.51 & -153.39 & 10039.34 \\
1913 & -49.96 & -56.22 & 12.53 & -1.62 & -73.63 & 4454.21 \\
1922 & -35.90 & -36.09 & 0.52 & -1.14 & -54.49 & 4693.84 \\
1931 & -39.42 & -69.16 & 75.44 & -1.18 & -91.16 & 7603.67 \\
1940 & -18.28 & -45.39 & 148.29 & 0.00 & -61.81 & $\infty$ \\
2445 & 324.32 & 321.45 & 0.88 & - & 188.67 & - \\
2698 & 1201.30 & 1190.90 & 0.87 & \textbf{18860.00} & \textbf{18671.64} & \textbf{1.00} \\
2834 & - & 365.00 & - & - & 311.32 & - \\
3089 & - & 312.62 & - & - & 215.77 & - \\
3225 & - & 511.28 & - & - & 466.67 & - \\
3385 & 587.78 & 582.06 & 0.97 & \textbf{225.00} & \textbf{225.00} & \textbf{0.00} \\
3387 & 408.32 & 396.39 & 2.92 & - & 275.84 & - \\
3814 & 0.63 & 0.63 & 0.00 & - & 0.29 & - \\
6287 & -2398.76 & -2422.72 & 1.00 & -2337.19 & -2511.26 & 7.45 \\
\bottomrule
\end{tabular}
\end{table}

Further inspection indicates that, in these instances, the variables involved in quadratic terms have relatively loose bounds, on the order of $10^6$. Such large ranges lead to a coarse and inaccurate \gls{cpwl} approximation. This highlights an important limitation of the proposed reformulation: as established in Theorems~\ref{th:approx_error} and~\ref{th:epsilon_n}, the accuracy of the method depends strongly on how tightly the variables appearing in each bilinear term are bounded.

These results clearly indicate that the proposed \gls{cpwl} reformulation is not a general-purpose replacement for \gls{qp} formulations. Rather, its computational benefits appear to be limited to specific problem classes, such as the \gls{scbp} problem considered above. For more general \gls{qp} problems, as illustrated by the QPLIB results, the reformulation can lead to both weaker optimality gaps and poor approximation quality, making the original \gls{qp} formulation preferable whenever it can be solved reliably.

Figure~\ref{fig:applicability} highlights the contrasting performances of the \gls{cpwl} reformulation on the \gls{scbp} and QPLIB instances. For the \gls{scbp} instances, the solve rate of the \gls{milp} reformulation is nearly twice that of the open-source \gls{qp} solver, and comparable to the commercial \gls{qp} solver. In contrast, for the QPLIB instances, the solve rate of the \gls{milp} reformulation is approximately five times lower than that of the open-source \gls{qp} solver.

\begin{figure}
    \centering
    \includegraphics[width=1.0\linewidth]{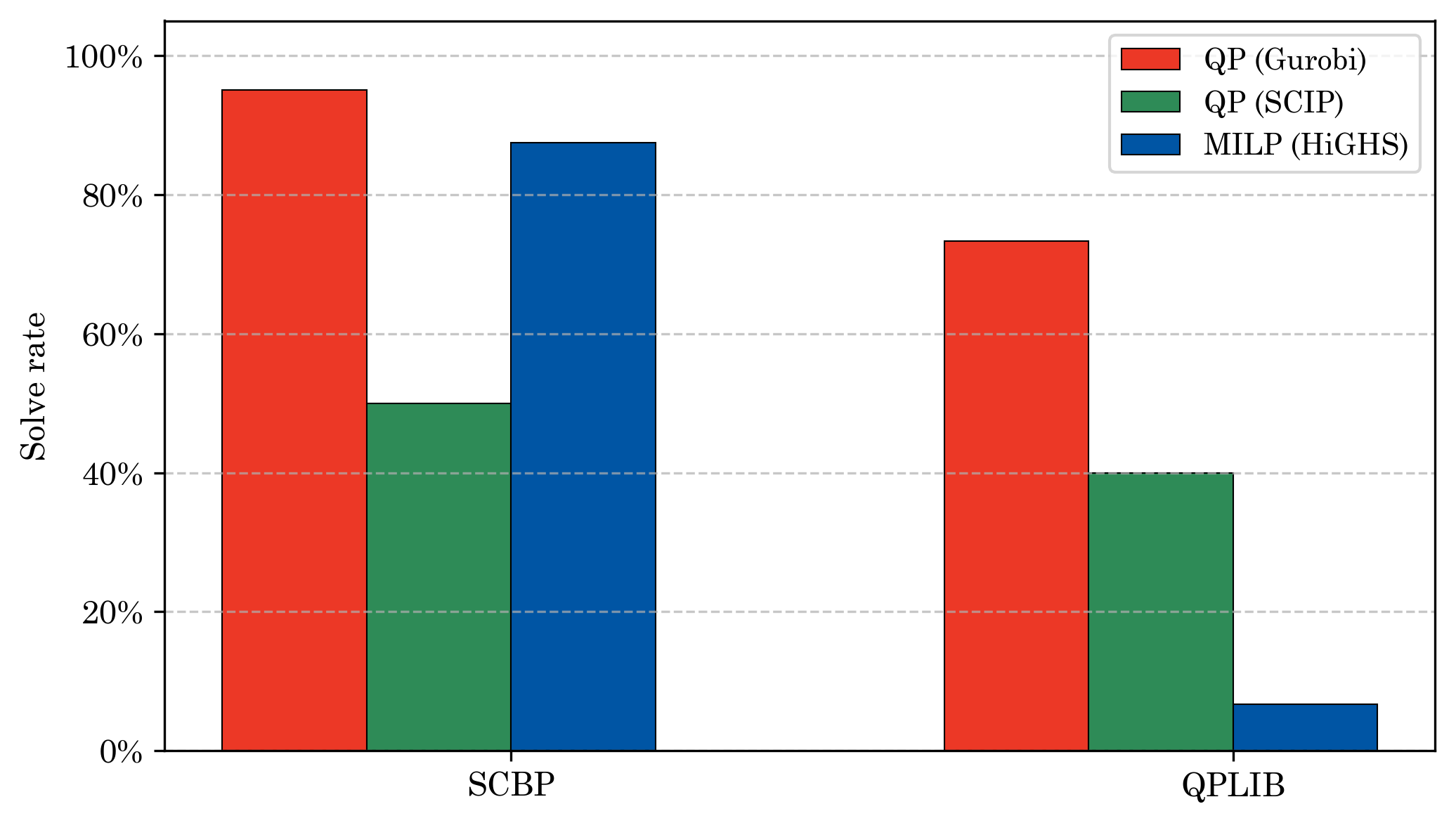}
    \caption[Comparison of solve rates between \gls{scbp} and QPLIB instances ($n=4$)]{Comparison of solve rates between \gls{scbp} and QPLIB instances ($n=4$)}
\label{fig:applicability}
\end{figure}

\section{Conclusion}

In this paper, we explored the specific conditions under which approximating nonconvex bilinear terms by \gls{cpwl} functions using \gls{milp} reformulations becomes highly advantageous over direct \gls{qp} methods. We demonstrated that the computational viability of this approach hinges on a distinct combination of the problem class, the tightness of the variable bounds, the chosen degree of accuracy, and the geometric structure of the \gls{milp} representation. To ground this comparison, we formalized the \gls{cpwl} function \( g_n \), analytically calculated its exact maximum and average approximation errors, and translated its geometry into three distinct \gls{milp} representations: the ``Triangle'', ``Square'', and ``\gls{dc}'' formulations.

Among these proposed formulations, our analyses reveal that the ``Square'' representation provides the optimal balance of compactness and relaxation tightness, outperforming the ``Triangle'' and ``\gls{dc}'' approaches in the vast majority of cases. However, testing on general QPLIB instances demonstrates that this \gls{milp} reformulation is not a universal replacement for direct \gls{qp} solvers. First, for problems with loose variable bounds, the \gls{cpwl} approximation becomes highly inaccurate, deviating significantly from the true \gls{qp} solution. Second, even when variable bounds are sufficiently tight, the \gls{milp} formulation fails to close the optimality gap on general QPLIB instances compared to state-of-the-art \gls{qp} solvers.

Conversely, these results highlight that the \gls{cpwl} reformulation is highly beneficial for specific problem classes like \gls{scbp} problems. For these problems, which inherently feature tight finite bounds and sequential state-coupling constraints, the ``Square'' \gls{milp} formulation scales exceptionally well. As the sequence horizon extends, this approach maintains a stable solve rate where traditional spatial branch-and-bound \gls{qp} solvers fail to converge. Ultimately, this geometric reformulation allows mature, open-source \gls{milp} solvers to efficiently tackle large-scale \gls{scbp} instances and outperform state-of-the-art commercial \gls{qp} solvers. 

Future work may explore extending these geometric approximation strategies to handle generalized indefinite quadratic functions or dynamic domain partitioning. Furthermore, several critical research gaps remain to be addressed. First, while we conjecture that the introduced function \( g_n \) represents the most efficient \gls{cpwl} approximation of the unit bilinear term with respect to the trade-off between the number of pieces and the approximation error, a formal mathematical proof of its optimality remains to be established. Second, a rigorous theoretical analysis is needed to thoroughly explain why the ``Square'' representation empirically outperforms the ``\gls{dc}'' formulation, despite the latter requiring fewer variables and constraints. Third, a stronger theoretical foundation is required to clarify the exact mathematical mechanisms that allow this \gls{milp} reformulation to scale better than direct \gls{qp} solvers on \gls{scbp} instances. Finally, identifying other distinct classes of \gls{qp} problems where this \gls{milp} reformulation yields a computational advantage over direct \gls{qp} solvers will be crucial for mapping the broader applicability and theoretical limitations of this framework.

\section*{Declarations}
\noindent\textbf{Funding} \\
This work was authored for the Department of Energy (DOE) Office of Critical Minerals and Energy Innovation by Argonne National Laboratory, operated by UChicago Argonne LLC under contract number DE-AC02-06CH11357. This study was supported by the HydroWIRES Initiative of DOE’s Hydropower and Hydrokinetics Office.

\vspace{0.5cm}

\noindent\textbf{Competing Interests} \\
The authors declare that they have no relevant financial or non-financial interests to disclose.

\vspace{0.5cm}

\noindent\textbf{Data Availability} \\
The datasets generated during and analyzed during the current study are available in the GitHub repository: \url{https://github.com/quentinplsrd/pwl-reformulations-of-xy}.

\vspace{0.5cm}

\noindent\textbf{Code Availability} \\
The code used for the computational experiments is available in the GitHub repository: \url{https://github.com/quentinplsrd/pwl-reformulations-of-xy}.

\vspace{0.5cm}

\noindent\fbox{%
    \begin{minipage}{\textwidth}
The submitted manuscript has been created by UChicago Argonne, LLC, Operator of Argonne National Laboratory (“Argonne”). Argonne, a U.S. Department of Energy Office of Science laboratory, is operated under Contract No. DE-AC02-06CH11357.
The U.S. Government retains for itself, and others acting on its behalf, a paid-up nonexclusive, irrevocable worldwide license in said article to reproduce, prepare derivative works, distribute copies to the public, and perform publicly and display publicly, by or on behalf of the Government. The Department of Energy will provide public access to these results of federally sponsored research in accordance with the DOE Public Access Plan. http://energy.gov/downloads/doe-public-access-plan
    \end{minipage}%
}

%


%
%

\bibliographystyle{spmpsci}      
\bibliography{references.bib}   


\end{document}


\maketitle

\begin{center}
  \textbf{Journal: Mathematical Programming Computation}
\end{center}

\section{Overview}

This document provides supplementary theoretical results for the article
``When MILP Beats QP: Piecewise-Linear Reformulations of Sequentially Coupled Bilinear Programs''

The supplementary material is organized as follows.
Section~\ref{sec:proofs} provides additional proofs.
Section~\ref{sec:formulation} describes the detailed MILP formulations of the ``Square'' and ``DC'' representation of the function $g_n$.

\section{Additional theoretical results}
\label{sec:proofs}

\subsection{Proof of co-planarity of the adjacent triangular pieces}

Let $(x_i,y_i)_{i = 1,2,3,4}$ be the four vertices of the square domain $D_{n,j,k}$. We prove here that the points $(x_i,y_i,x_iy_i)_{i=1,2,3,4} = \left\{ P_1, P_2, P_3, P_4 \right\}$ are co-planar.

\begin{proof}
    Let $(x_0,y_0)$ be center of the square domain $D_{n,j,k}$. We have: 
    \[
    (x_0,y_0) = \left(\frac{j-k+n}{2n},\frac{j+k+1-n}{2n}\right)
    \]
    The four vertices of $D_{n,j,k}$ are:
    \begin{align*}
        (x_1,y_1) &= \left(x_0 - \frac{1}{2n}, y_0\right) \\
        (x_2,y_2) &= \left(x_0, y_0 + \frac{1}{2n}\right) \\
        (x_3,y_3) &= \left(x_0, y_0 - \frac{1}{2n}\right) \\
        (x_4,y_4) &= \left(x_0 + \frac{1}{2n}, y_0\right)
    \end{align*}

   We can show that the four points $P_i$ are co-planar by showing that the vectors $\overrightarrow{P_1P_2}$ and $\overrightarrow{P_3P_4}$ are identical.
    
    \[
    \overrightarrow{P_1P_2} 
    = \left[ \begin{array}{c} x_2 - x_1 \\ y_2 - y_1 \\ x_2y_2 - x_1y_1 \end{array} \right]
    = \left[ \begin{array}{c} x_0 - x_0 + \frac{1}{2n} \\ y_0 + \frac{1}{2n} - y_0 \\ x_0y_0 + \frac{1}{2n}x_0 - x_0y_0 + \frac{1}{2n}y_0 \end{array} \right]
    = \left[ \begin{array}{c} \frac{1}{2n} \\ \frac{1}{2n} \\ \frac{1}{2n}(x_0+y_0) \end{array} \right]
    \]
    
    \[
    \overrightarrow{P_3P_4} 
    = \left[ \begin{array}{c} x_4 - x_3 \\ y_4 - y_3 \\ x_4y_4 - x_3y_3 \end{array} \right]
    = \left[ \begin{array}{c} x_0 + \frac{1}{2n} - x_0  \\ y_0  - y_0 + \frac{1}{2n} \\ x_0y_0 + \frac{1}{2n}y_0 - x_0y_0 + \frac{1}{2n}x_0 \end{array} \right]
    = \left[ \begin{array}{c} \frac{1}{2n} \\ \frac{1}{2n} \\ \frac{1}{2n}(x_0+y_0) \end{array} \right]
    \]
\qed
\end{proof}

\subsection{Proof of the expression of $g_n$ on $D_{n,j,k}$}
\begin{proof}
First, we calculate $\overrightarrow{u}$, a normal vector to the plane $\bm{\Pi}$  defined by the four points $P_1$,$P_2$,$P_3$, and $P_4$. 
\begin{align*}
\overrightarrow{u} &= \overrightarrow{P_1P_2} \times \overrightarrow{P_1P_3} = \left[ \begin{array}{c} \frac{1}{2n} \\\frac{1}{2n} \\ \frac{1}{2n}(x_0+y_0) \end{array} \right] \times \left[ \begin{array}{c} \frac{1}{2n} \\ -\frac{1}{2n} \\ \frac{1}{2n}(y_0-x_0) \end{array} \right]
\\&= \frac{1}{(2n)^2}\left[ \begin{array}{c} (y_0-x_0)+(x_0+y_0) \\ (x_0+y_0) - (y_0-x_0) \\ -1-1 \end{array} \right]
= \frac{1}{2n^2}\left[ \begin{array}{c} y_0 \\ x_0 \\ -1 \end{array} \right]
\end{align*}

Therefore, the equation of the plane $\bm{\Pi}$ is:
\[
y_0x+x_0y-z=C
\]
Because $P_1 \in \bm{\Pi}$, we have:
\begin{align*}
&y_0x_1+x_0y_1-x_1y_1=C \\
\Leftrightarrow \enspace &y_0\left(x_0-\frac{1}{2n}\right)+x_0y_0-\left(x_0-\frac{1}{2n}\right)y_0=C \\
\Leftrightarrow \enspace &C = x_0y_0
\end{align*}

As a result, the expression of $g_n$ on $D_{n,j,k}$ is:
\begin{align*}
&g_n(x,y) = y_0x+x_0y - x_0y_0 \\
= &\left(\frac{j+k+1-n}{2n}\right)x + \left(\frac{j-k+n}{2n}\right)y - \left(\frac{j-k+n}{2n}\right)\left(\frac{j+k+1-n}{2n}\right)
\end{align*}
\qed
\end{proof}

\subsection{Proof of the DC expression of $g_n$}

\begin{proof} Let $g_n^+$ and $g_n^-$ be defined as:
\begin{align*}
&g_n^+(x,y)
= \max_{0 \leq j < 2n} g_{n,j}^+(x,y) \\
&= \max_{0 \leq j < 2n}
\left[
\left( \frac{2j+1}{4n}\right)x
+
\left( \frac{2j+1}{4n}\right)y
-
\frac{j(j+1)}{4n^2}
\right] \\
&g_n^-(x,y)
= \max_{0 \leq k < 2n} g_{n,k}^-(x,y) \\
&= \max_{0 \leq k < 2n}
\left[
-\left( \frac{2(k-n)+1}{4n} \right)x
+
\left( \frac{2(k-n)+1}{4n} \right)y
-
\frac{(k-n)(k-n+1)}{4n^2}
\right]
\end{align*}
We want to show that 
\[
g_n(x,y) = g_n^+(x,y) - g_n^-(x,y), \enspace \forall (x,y) \in [0,1]^2
\]
First, we show that 
\begin{align*}
g_n^+(x,y) &= g_{n,j}^+(x,y), \enspace \forall (x,y) \in D_{n,j}^+=\left\{(x,y) \in [0,1]^2: \frac{j}{n} \leq x+y \leq \frac{j+1}{n}\right\} \\
g_n^-(x,y) &= g_{n,k}^-(x,y), \enspace \forall (x,y) \in D_{n,k}^-=\left\{(x,y) \in [0,1]^2: \frac{k}{n} \leq y-x+1 \leq \frac{k+1}{n}\right\}
\end{align*}

Let $j \in \{1, ..., 2n-1\}$.
\begin{align*}
g_{n,j}^+(x,y) &\geq g_{n,j-1}^+(x,y) \\
 \Leftrightarrow \enspace \frac{2}{4n} x + \frac{2}{4n} y &\geq \frac{j(j+1) - j(j-1)}{4n^2} \\
\Leftrightarrow \enspace \frac{1}{2n} x + \frac{1}{2n} y &\geq \frac{2j}{4n^2} \\
\Leftrightarrow \enspace x + y &\geq \frac{j}{n}
\end{align*}

Let $k \in \{1, ..., 2n-1\}$.
\begin{align*}
g_{n,k}^-(x,y) &\geq g_{n,k-1}^-(x,y) \\
 \Leftrightarrow \enspace -\frac{2}{4n} x + \frac{2}{4n} y &\geq \frac{(k-n)(k-n+1) - (k-n)(k-n-1)}{4n^2} \\
\Leftrightarrow \enspace -\frac{1}{2n} x + \frac{1}{2n} y &\geq \frac{2(k - n)}{4n^2} \\
\Leftrightarrow \enspace y-x &\geq \frac{k}{n} -1
\end{align*}

Therefore, $g_n^+$ is equal to $g_{n,j}^+$ on the domain $D_{n,j}^+$, and $g_n^-$ is equal to $g_{n,k}^-$ on the domain $D_{n,k}^-$.

Let $(x,y) \in [0,1]^2$. Note that $[0,1]^2 = \bigcup_{0 \leq j < 2n}{D_{n,j}^+} = \bigcup_{0 \leq k < 2n}{D_{n,k}^-}$. This implies that $\exists (j,k):(x,y) \in D_{n,j}^+ \bigcap D_{n,k}^- = D_{n,j,k}$.

It follows that
\begin{align*}
&g_n^+(x,y) - g_n^-(x,y) = g_{n,j}^+(x,y) - g_{n,k}^-(x,y) \\
&= \left(
\left( \frac{2j+1}{4n}\right)x
+
\left( \frac{2j+1}{4n}\right)y
-
\frac{j(j+1)}{4n^2}
\right) \\
&-
\left(
-\left( \frac{2(k-n)+1}{4n} \right)x
+
\left( \frac{2(k-n)+1}{4n} \right)y
-
\frac{(k-n)(k-n+1)}{4n^2}
\right) \\
&= \left(\frac{j+k+1-n}{2n}\right)x + \left(\frac{j-k+n}{2n}\right)y - \left(\frac{j-k+n}{2n}\right)\left(\frac{j+k+1-n}{2n}\right) \\
&= g_n(x,y)
\end{align*}
\qed    
\end{proof}

\section{Mathematical formulations}
\label{sec:formulation}

\subsection{MILP formulation of the ``Square'' representation of $g_n$}

Let $\mathcal{P} = \left\{ (j,k) \in \{0, ..., 2n-1\}^2 : int(D_{n,j,k}) \neq \emptyset \right\}$. The MILP formulation of $z=g_n(x,y)$ under the ``Square'' representation is described by the system of linear equations below.

\begin{align}
z = &\sum_{(j,k) \in \mathcal{P}} {z_{j,k}}\\
\begin{bmatrix} x \\ y \end{bmatrix} = &\sum_{(j,k) \in \mathcal{P}} {\begin{bmatrix} x_{j,k} \\ y_{j,k} \end{bmatrix}}\\
z_{j,k} = & \frac{j+k+1-n}{2n} x_{j,k} \nonumber \\
&+  \frac{j-k+n}{2n} y_{j,k}  \nonumber \\
&-\delta_{j,k} \frac{(j+k+1-n)(j-k+n)}{4n^2} , & (j,k) \in \mathcal{P}\\
\begin{bmatrix} 
x_{j,k}+y_{j,k} \\
-x_{j,k}-y_{j,k} \\
-x_{j,k}+y_{j,k} \\
x_{j,k}-y_{j,k}
\end{bmatrix}
\leq &\delta_{j,k} 
\begin{bmatrix} 
\frac{j+1}{n} \\
-\frac{j}{n} \\
\frac{k+1-n}{n} \\
-\frac{k-n}{n}
\end{bmatrix}, &(j,k) \in \mathcal{P}\\
\sum_{(j,k) \in \mathcal{P}} {\delta_{j,k}} = &1\\
x, y, z \in &[0,1]\\
x_{j,k}, y_{j,k}, z_{j,k} \in &[0,1], &(j,k) \in \mathcal{P}\\
\delta_{j,k} \in &\{0,1\}, &(j,k) \in \mathcal{P}
\end{align}

Equation S3.1.1, and S3.1.2, represent the decomposition of the variables $x$, $y$, and $z$ into their piecewise linear components. Equation S3.1.3 represents the affine equation of each linear piece. Equation S3.1.4 represents the square-shaped domain of each linear piece. Equation S3.1.5 indicates that only one linear piece can be activated at a time. 

\subsection{MILP formulation of the ``DC'' representation of $g_n$}

The MILP formulation of $z=g_n(x,y)$ under the ``DC'' representation is described by the system of linear equations below.

\begin{align}
z = &z^+ - z^- \\
z^c = & \sum_{j=0}^{2n-1}{z_j^c}, &c \in \{+,-\}\\
\begin{bmatrix} x \\ y \end{bmatrix} = &\sum_{j=0}^{2n-1} {
\begin{bmatrix} 
x_{j}^c \\ 
y_{j}^c 
\end{bmatrix}
}, &c \in \{+,-\}\\
z_j^+ = & \frac{2j+1}{4n} x_j^+
\nonumber \\
&+ \frac{2j+1}{4n} y_j^+
\nonumber \\
&- \delta_j^+\frac{j(j+1)}{4n^2},
& j \in \{0,...,2n-1\}\\
z_k^- = &- \frac{2(k-n)+1}{4n} x_k^-
\nonumber \\
&+  \frac{2(k-n)+1}{4n} y_k^-
 \nonumber \\
 &- \delta_k^-\frac{(k-n)(k-n+1)}{4n^2},
& k \in \{0,...,2n-1\}\\
\begin{bmatrix} 
x_{j}^++y_{j}^+ \\
-x_{j}^+-y_{j}^+ 
\end{bmatrix}
\leq &\delta_{j}^+ 
\begin{bmatrix} 
\frac{j+1}{n} \\
-\frac{j}{n} 
\end{bmatrix}, &j \in \{0,...,2n-1\}\\
\begin{bmatrix} 
-x_{k}^-+y_{k}^- \\
x_{k}^--y_{k}^- 
\end{bmatrix}
\leq &\delta_{k}^- 
\begin{bmatrix} 
\frac{k+1-n}{n} \\
-\frac{k-n}{n} 
\end{bmatrix}, &k \in \{0,...,2n-1\}\\
\begin{bmatrix} 
x_{k}^- \\
y_{k}^- 
\end{bmatrix}
\leq &\delta_{k}^- 
\begin{bmatrix} 
1 \\
1 
\end{bmatrix}, &k \in \{0,...,2n-1\}\\
\sum_{j=0}^{2n-1} {\delta_j^c} = &1, &c \in \{+,-\} \\
x, y, z \in &[0,1]\\
x_j^+, y_j^+, z_j^+ \in &[0,1], &j \in \{0,...,2n-1\}\\
x_k^-, y_k^-, z_k^- \in &[0,1], &k \in \{0,...,2n-1\}\\
\delta_j^+, \delta_k^- \in &\{0,1\}, &j,k \in \{0,...,2n-1\}
\end{align}

Equation S3.2.1 represents the variable $z$ as the difference of two convex components $z^+$ and $z^-$. Equation S3.2.2 represents the decomposition of each convex component of $z$ into its piecewise linear components, and Equation S3.2.3 represents the decomposition of the variables $x$ and $y$ into the piecewise linear components of their convex components. Equations S3.2.4 and S3.2.5 represent the affine equation of each linear piece of $g_n^+$ and $g_n^-$. Equations S3.2.6 and S3.2.7 represent the strip-shaped domain of each linear piece of $g_n^+$ and $g_n^-$. Equation S3.2.8 is required to ensure that the vector $[x_k^-,y_k^-]$ collapses to $[0,0]$ when the $k^{th}$ linear piece of $g_n^-$ is not active. Equation S3.2.9 indicates that only one linear piece of $g_n^+$, $g_n^-$, can be activated at a time.
